\documentclass[12pt, reqno]{amsart}
\usepackage{amssymb, hyperref}
\usepackage{amsthm}
\usepackage{amsmath}
\usepackage{graphicx}
\usepackage[all,2cell]{xy}
\usepackage{verbatim}
\usepackage{tikz}
\usepackage{tikz-cd}
\usepackage{hyperref}
\usepackage{array}
\usepackage[capitalise]{cleveref}
\usepackage{epsfig,graphicx}
\usepackage{tikz}
\usepackage{tikz-cd}
\tikzset{node distance=2cm, auto}
\usepackage[vcentermath]{youngtab}
\usepackage{ytableau}
\usepackage{hyperref}
\usetikzlibrary{graphs}
\usetikzlibrary{backgrounds}
\usepackage{multicol}
\usetikzlibrary{arrows,decorations.markings}
\tikzstyle{vertex}=[circle, draw, inner sep=0pt, minimum size=6pt]
\newcommand{\vertex}{\node[vertex]}
\usepackage{setspace}
\usetikzlibrary{matrix,arrows,decorations.pathmorphing}
\usepackage{mathrsfs}
\numberwithin{equation}{section}
\newtheorem*{theorem*}{Theorem}
\newtheorem*{corollary*}{\bf Corollary}
\newtheorem*{remark*}{\bf Remark}
\usepackage{lipsum}
\usepackage{lineno}
\newtheorem{theorem}{Theorem}[section]
\newtheorem{corollary}[theorem]{Corollary}

\newtheorem{lemma}[theorem]{Lemma}

\title[Torus quotient of Richardson varieties]
{Torus quotient of Richardson varieties in Orthogonal and Symplectic Grassmannians}
\newtheorem{proposition}[theorem]{Proposition}
\newtheorem{remark}[theorem]{Remark}
\title[Projective normality of torus quotients of flag varieties]
{Projective normality of torus quotients of flag varieties}
 \author{Arpita Nayek}
\address{%
Arpita Nayek\\
Department of Mathematics and Statistics \\
Indian Institute of Technology, Kanpur\\
Kanpur-208016 \\
India\\
Email: anayek@iitk.ac.in\\
}
\author{S.K. Pattanayak}
\address{%
S.K. Pattanayak\\
Department of Mathematics and Statistic\\
Indian Institute of Technology, Kanpur\\
Kanpur-208016\\
India\\
Email:santosha@iitk.ac.in\\
}
\author{Shivang Jindal}
\address{%
Shivang Jindal\\
Department of Mathematics and Statistic\\
Indian Institute of Technology, Kanpur\\
Kanpur-208016\\
India\\
Email:keshushivang@gmail.com\\
}


\begin{document}
\maketitle
\begin{abstract}
Let $G=SL_n(\mathbb C)$ and $T$ be a maximal torus in $G$. We show that the  quotient $T \backslash \backslash G/{P_{\alpha_1}\cap P_{\alpha_2}}$ is projectively normal with respect to the descent of a suitable line bundle, where $P_{\alpha_i}$ is the maximal parabolic subgroup in $G$ associated to the simple root $\alpha_i$, $i=1,2$.  We give a degree bound of the generators of the homogeneous coordinate ring of $T \backslash \backslash (G_{3,6})^{ss}_T(\mathcal{L}_{2\varpi_3})$. If $G =Spin_7$, we give a degree bound of the generators of the homogeneous coordinate ring of $T \backslash \backslash (G/P_{\alpha_2})^{ss}_T(\mathcal{L}_{2\varpi_2})$ whereas we prove that the quotient $T\backslash\backslash (G/P_{\alpha_3})^{ss}_T(\mathcal{L}_{4\varpi_3})$ is projectively normal with respect to the descent of the line bundles $\mathcal{L}_{4\varpi_3}$. 
\end{abstract}


\section{Introduction}\label{s.Introduction}

For the action of a maximal torus $T$ on the Grassmannian $G_{r,n}$ the quotients $T \backslash\backslash G_{r,n}$ have been studied extensively. Allen Knutson called them weight varieties in his thesis \cite{Knu}. In \cite{HK} Hausmann and Knutson identified the GIT quotient of the Grassmannian $G_{2,n}$ by the natural action of  the maximal torus with the moduli space of polygons in $\mathbb R^3$ and this GIT quotient can also be realized as  the GIT quotient of an $n$-fold product of projective lines by the diagonal action of $PSL(2, \mathbb C)$. In the symplectic geometry literature these spaces are known as polygon spaces as they parameterize the $n$-sides polygons in $\mathbb R^3$ with fixed edge length up to rotation. More generally, $T \backslash\backslash G_{r,n}$ can be identified with the GIT quotient of $(\mathbb P^{r-1})^{n}$ by the diagonal action of $PSL(r, \mathbb C)$ via the Gelfand-MacPherson correspondence. In \cite{Kap1} and \cite{Kap2} Kapranov studied the Chow quotient of the Gassmannians and he showed that the Grothendieck-Knudsen moduli space $\overline{M}_{0,n}$ of stable $n$-pointed curves of genus zero arises as the Chow quotient of the maximal torus action on the Grassmannian $G_{2,n}$.

In \cite{D} Dabrowski has proved that for any parabolic subgroup $P$ of $G$, the Zariski closure of a generic $T$-orbit in $G/P$ is normal.  For a precise statement, see \cite[Theorem 3.2, pg.~327]{D}. In \cite{CK} Carrell and Kurth proved that if $G$ is of type $A_n$, $D_4$ or $B_2$ and $P$ is any maximal parabolic subgroup of $G$, then every $T$ orbit closure in $G/P$ is normal. In the context of a problem on projective normality for torus actions, Howard proved that for any parabolic subgroup $P$ of $SL_n(\mathbb C)$, the Zariski closure $\overline{T.x}$ of the $T$-orbit of any point $x$ in $SL_n(\mathbb C)/P$ is projectively normal for the choice of any ample line bundle $\mathcal L$ on $SL_n(\mathbb C)/P$. For a precise statement, see \cite[Theorem 5.4, pg.~540]{How}.

In \cite{KPS}  the authors consider the quotients of a projective space $X$ for the linear action of finite solvable groups and for finite groups acting by pseudo reflections. They prove that the descent of $\mathcal O_X(1)^{|G|}$ is projectively normal. In \cite{CHK} these results were obtained for every finite group but with a larger power of the descent of $\mathcal O_X(1)^{|G|}$. In \cite{K1} there was an attempt to study the projective normality of $T\backslash\backslash(G_{2,n})$ ($n$ odd) with respect to the descent of the line bundle corresponding to the fundamental weight $\omega_2$. There it was proved that the homogeneous coordinate ring of $T\backslash\backslash(G_{2,n})$ is a finite module over the subring generated by the degree one elements. In \cite{sarjick} and \cite{HMSV} the authors show that the quotient $T \backslash \backslash G_{2,n}$ is projectively normal with respect to the descent of the line bundle corresponding to $n\varpi_2$. 

In this paper we give a short proof of the projective normality of the quotient $T \backslash \backslash G_{2,n}$ with respect to the descent of the line bundle $\mathcal{L}_{n\varpi_2}$ using Standard Monomial Theory and some graph theoretic techniques. We also prove that the quotient $T \backslash \backslash G/{P_{\alpha_1}\cap P_{\alpha_2}}$ is projectively normal with respect to the descent of a suitable line bundle, where $P_{\alpha_i}$ is the maximal parabolic subgroup in $G$ associated to the simple root $\alpha_i$, $i=1,2$, which is the main ingredient of this paper. We give a degree bound of the generators of the homogeneous coordinate ring of $T \backslash \backslash (G_{3,6})^{ss}_T(\mathcal{L}_{2\varpi_3})$. If $G = Spin_7$, we give a degree bound of the generators of the homogeneous coordinate ring of  $T \backslash \backslash (G/P_{\alpha_2})^{ss}_T(\mathcal{L}_{2\varpi_2})$ whereas we prove that $T\backslash\backslash (G/P_{\alpha_3})^{ss}_T(\mathcal{L}_{4\varpi_3})$ is projectively normal with respect to the descent of the line bundles $\mathcal{L}_{4\varpi_3}$. 

The layout of the paper is as follows. \cref{s.Preliminaries} consists of preliminary definitions and notation. In \cref{smt} we recall some preliminaries of Standard Monomial Theory and in \cref{graphtheory} we recall some preliminaries of graph theory. In \cref{isomorphism} we show that the GIT quotients $T \backslash\backslash(G_{r,n})$ and $T \backslash\backslash(G_{n-r,n})$ are isomorphic. In \cref{grassmann} we give a proof of the projective normality of the quotient $T \backslash \backslash (G_{2,n})^{ss}_T(\mathcal{L}_{n\varpi_2})$ with respect to the descent of the line bundle $\mathcal{L}_{n\varpi_2}$ and we give a degree bound of the generators of the homogeneous coordinate ring of $T \backslash \backslash (G_{3,6})^{ss}_T(\mathcal{L}_{2\varpi_3})$. In \cref{partialflag}  for $G=SL_n$ we prove projective normality of the quotient $T \backslash \backslash G/{P_{\alpha_1}\cap P_{\alpha_2}}$ with respect to the descent of a suitable line bundle and in \cref{secb3} for $G=Spin_7$, we give a degree bound of the generators of the homogeneous coordinate ring of  $T \backslash \backslash (G/P_{\alpha_2})^{ss}_T(\mathcal{L}_{2\varpi_2})$ and we prove that $T\backslash\backslash (G/P_{\alpha_3})^{ss}_T(\mathcal{L}_{4\varpi_3})$ is projectively normal with respect to the descent of the line bundle $\mathcal{L}_{4\varpi_3}$.

 \section{Preliminaries}\label{s.Preliminaries}

 In this section we set up some preliminaries and notation. We refer to \cite{Hum1}, \cite{Hum2}, \cite{spr} for preliminaries in Lie algebras and algebraic groups. Let $G$ be a semi-simple algebraic group over $\mathbb C$. We fix a maximal torus $T$ of $G$ and a Borel subgroup $B$ of $G$ containing $T$.  Let $N_G(T)$ be the normaliser of $T$ in $G$. Let $W = N_G(T)/T$ be the Weyl group of $G$ with respect to $T$. Let $R$ denotes the set of roots with respect to $T$. Let $S = \{\alpha_1,\alpha_2,\ldots,\alpha_n\} \subset R$ be the set of simple roots and for a subset $I\subseteq S$ we denote by $P_I$  the parabolic subgroup of $G$ generated by $B$ and $\{n_{\alpha}: \alpha \in I^{c}\}$, where $n_{\alpha}$ is a representative of $s_{\alpha}$ in $N_{G}(T)$. Let $X(T)$ (resp. $Y(T)$) denote the set of characters of $T$ (resp. one parameter subgroups of $T$). Let $E_1 := X(T) \otimes \mathbb{R}$,  $E_2 := Y(T) \otimes \mathbb{R}$. Let $\langle.,.\rangle: E_1 \times E_2 \to \mathbb{R}$ be the canonical non-degenerate bi-linear form. For all homomorphism $\phi_\alpha : SL_2 \to G$, $(\alpha \in R)$, we have $\check{\alpha}: \mathbb{G}_m \to G$ defined by 
	  
	   \[\check{\alpha}(t)=\phi_{\alpha} ( \left[ {\begin{array}{cc}
   t & 0\\
   0 & t^{-1} \\
  \end{array} } \right]).\] We also have $s_{\alpha}(\chi) = \chi - \langle\chi,\check{\alpha}\rangle \alpha$ for all $\alpha \in R$ and $\chi \in E_1$. Set $s_i = s_{\alpha_i}$ for all $i = 1, 2, \ldots,n$. Let $\{\varpi_i: i=1, 2, \ldots, n\} \in E_1$ be the fundamental weights; i.e.   
           $\langle\varpi_i,\check{\alpha_j}\rangle = \delta_{ij}$ for all $i,j = 1,2,\ldots n.$

 For a simply connected semi-simple algebraic group $G$ and for a parabolic subgroup $P$, the quotient space $G/P$ is a homogenous space for the left action of $G$. The quotient $G/P$ is called a generalized flag variety. When $G=SL_n(\mathbb C)$ and $P_r$ is the maximal parabolic subgroup corresponding to the simple root $\alpha_r$, the quotient can be identified with $G_{r,n}$, the Grassmannian of $r$ dimensional subspaces of $\mathbb C^n$. 

Now we recall the definition of projective normality of a projective variety. A projective variety $X \subset \mathbb{P}^n$ is said to be projectively normal if the affine cone $\hat{X}$ over $X$ is normal at its vertex. For a reference, see exercise 3.18, page 23 of \cite{RH}. For the practical purpose we need the following fact about projective normality of a polarized variety. 

%\textcolor{red}{Projective normality of a projective variety depends on the particular projective embedding in a projective space.}

A polarized variety $(X, \mathcal L)$ where $\mathcal L$ is a very ample line 
bundle is said to be projectively normal if its homogeneous coordinate ring 
$\oplus_{n \in \mathbb Z_{\geq 0}}H^0(X, \mathcal L^{\otimes n})$ is integrally 
closed and it is generated as a $\mathbb{C}$-algebra by $H^0(X, \mathcal L)$ (see Exercise 5.14, Chapter II of \cite{RH}). Projective normality depends on the particular projective embedding of the variety.

\textbf{Example:} The projective line $\mathbb{P}^1$ is obviously projectively normal since its cone is the affine plane $\mathbb C^2$ (which is non-singular). However it can also be embedded in $\mathbb{P}^3$ as the quartic curve, namely,\\
\centerline{$V_{+}=\{(a^4,a^3b,ab^3,b^4) \in \mathbb{P}^3 ~|~ (a,b) \in \mathbb{P}^1\}$,}\\ 
then it is normal but not projectively normal (see \cite{RH}, Chapter 1. Ex. 3.18).

Let $X$ be a projective variety which is acted upon by a reductive group $G$. Let $\mathcal L$ be a $G$-linearized very ample line bundle on $X$. The GIT quotient $X//G$ is by definition the uniform categorical quotient of the (open) set of semistable points $X^{ss}_G(\mathcal{L})$ by $G$. We denote the GIT quotient of $X$ by $G$ with respect to $\mathcal{L}$ by $X^{ss}_G(\mathcal{L})//G$. Assume that the line bundle $\mathcal L$ descends to the quotient $X^{ss}_G(\mathcal{L})//G$ and denote the descent by $\mathcal L'$. Then the polarized variety ($X^{ss}_G(\mathcal{L})//G, \mathcal L'$) is $Proj(\oplus_{n \in \mathbb{Z}_\geq0}(H^0(X,\mathcal{L}^{\otimes n})^G)$. For preliminaries in Geometric Invariant theory we refer to \cite{MFK} and \cite{PN}.

%For standard monomial theory on the Grassmannian $G_{r,n}$ we refer to \cite{LB}.

%\begin{theorem}
%The standard monomials of degree $m$ form a basis of the vector space of degree $m$ homogeneous polynomials in the homogeneous coordinate ring of $G_{r,n}$. 
%\end{theorem}

 Let $G$ be a simple, simply-connected algebraic group of type $A$ or $B$. Let $T$ be a maximal torus in $G$ and $Q$ be the root lattice of $G$. Let $\lambda$ be a dominant weight of $G$ and $P_{\lambda}$ be the parabolic subgroup of $G$ associated to $\lambda$. Let $\mathcal{L}_{\lambda}$ be the homogeneous ample line bundle on $G/P_{\lambda}$ associated to $\lambda$. Then the following theorem describes which line bundles descend to the GIT  quotient $T \backslash\backslash (G/P_{\lambda})^{ss}_T(\mathcal{L}_\lambda)$ (see \cite[Theorem 3.10]{KS}). 
 
 \begin{theorem}\label{shrawan} With all the notations as above, the line bundle $\mathcal{L}_{\lambda}$ descends to a line bundle on the GIT quotient $T \backslash\backslash (G/P_{\lambda})^{ss}_T(\mathcal{L}_\lambda)$ if and only if $\lambda$ is of the following form depending upon the type $G$:
 
   1. $G$ of type $A_n (n \geq 1)$: $\lambda \in Q$,
   
   2. $G$ of type $B_2$: $\lambda \in \mathbb{Z}\alpha_1+2\mathbb{Z}\alpha_2$,
    
   3. $G$ of type $B_n(n \geq 3)$: $\lambda \in 2Q$.
   
    \end{theorem} 
    
    \section{Some preliminaries of Standard Monomial Theory}\label{smt}
    
    Let $\{e_1, e_2, \ldots, e_n\}$ be the standard basis of $\mathbb{C}^n$. Let $I_{r,n} = \{(i_1,i_2,\ldots,i_r)| 1 \leq i_1 < \cdots < i_r \leq n\}$. The set $\{e_{i_1}\wedge e_{i_2} \wedge \ldots \wedge e_{i_r}| (i_1,i_2, \ldots, i_r) \in I_{r,n})\}$ form a basis of $\wedge^r \mathbb{C}^n$. We denote by $\{p_{i_1,i_2,\ldots,i_r}\}$ the dual basis of the basis $\{e_{i_1}\wedge e_{i_2} \wedge \ldots \wedge e_{i_r}\}$; the $p_{i_1,i_2,\ldots,i_r}$ are called the {\em Pl\"{u}cker coodinates} of $\mathbb{P}(\wedge^r \mathbb{C}^n)$.
    
    The Grassmannian $G_{r,n} \subseteq \mathbb{P}(\wedge^r \mathbb{C}^n)$ is precisely the zero set of the following well known Pl\"{u}cker relations:
    
     $$\sum_{h=1}^{r+1}(-1)^h p_{i_1,i_2,\ldots,i_{r-1}j_h}p_{j_1,\ldots,\hat{j_h},\ldots,j_{r+1}},$$ 
     
     where $\{i_1,\ldots, i_{r-1}\}$ and $\{j_1,\ldots,j_{r+1}\}$ are two subsets of $\{1,2,\ldots,n\}$.
    
    \subsection{$SL_n$-standard Young tableau}
    
    In this subsection we recall some basic facts about standard Young tableau for generalized flag varieties (see \cite[pg.~216]{LB}).
    
     Let $G = SL_n$ and $\lambda=\Sigma_{i=1}^{n-1} a_i\varpi_i$, $a_i \in \mathbb{Z}^{+}$ be a dominant weight. To $\lambda$ we associate a Young diagram (denoted by $\Gamma$) with $\lambda_i$ number of boxes in the $i$-th column, where $\lambda_i:=a_i+\ldots+a_{n-1}$, $1 \leq i \leq n-1$.
     
     	A Young diagram $\Gamma$ associated to a dominant weight $\lambda$ is said to be a Young tableau if the diagram is filled with integers $1, 2, \ldots, n$. We also denote this Young tableau by $\Gamma$. A Young tableau is said to be standard if the entries along any column is non-decreasing and along any row is strictly increasing. 
     	
     	Given a Young tableau $\Gamma$, let $\tau=\{i_1,i_2,\ldots,i_d\}$ be a typical row in $\Gamma$, where $1 \leq i_1 < \cdots < i_d \leq n$, for some $1 \leq d \leq n-1$. To the row $\tau$, we associate the Pl\"{u}cker coordinate $p_{i_1,i_2, \ldots,i_d}$. We set $p_{\Gamma}=\prod_{\tau}p_{\tau}$, where the product is taken over all the rows of $\Gamma$. We say that $p_{\Gamma}$ is a standard monomial on $G/P_{\lambda}$ if $\Gamma$ is standard, where $P_{\lambda}$ is the parabolic subgroup of $G$ associated to the weight $\lambda$.

Now we recall the definition of weight of a standard Young tableau $\Gamma$ (see \cite[Section 2]{LP}). For a positive integer $1\leq i\leq n$, we denote by $c_{\Gamma}(i)$,
the number of boxes of $\Gamma$ containing the integer $i$. Let $\epsilon_i: T\to \mathbb G_m$  
be the character defined as $\epsilon_i(diag(t_1,\ldots, t_n))=t_i$. 
%Note that the $i^{th}$ fundamental weight $\omega_i=\epsilon_1+ \cdots +\epsilon_i$.
 We define the weight of $\Gamma$ as 
$$wt(\Gamma):=c_{\Gamma}(1)\epsilon_1+ \cdots + c_{\Gamma}(n)\epsilon_n.$$

 We have the following lemma about $T$-invariant monomials in $H^0(G/P_{\lambda}, \mathcal L_{\lambda})$.
  
\begin{lemma}\label{6.2}
 A monomial $p_{\Gamma}\in H^0(G/P_{\lambda}, \mathcal L_{\lambda})$ is $T$-invariant if and only if  all the entries in $\Gamma$ appear equal number of times. 
\end{lemma}
\begin{proof}
 Recall that the action of $T$ on $H^0(G/P_{\lambda}, \mathcal L_{\lambda})$ is given by $$(t_1, \ldots, t_n)\cdot p_{i_1,i_2,\ldots, i_r}=(t_{i_1}\cdots t_{i_r})^{-1}p_{i_1,i_2, \ldots, i_r}.$$
Since $p_{i_1,\ldots, i_r}$ is the dual of $e_{i_1}\wedge \cdots \wedge e_{i_r}$, the weight of $p_{i_1,\ldots, i_r}$ is $-(\epsilon_{i_1}+\cdots +\epsilon_{i_r})$.
Thus the weight of $p_{\Gamma}$ is $-wt(\Gamma)$.
Therefore, we see that $p_{\Gamma}$ is $T$ invariant if and only if the weight of $\Gamma$ is zero.
Since the weight of $\Gamma$ is $\sum_{i=1}^{n}c_{\Gamma}(i)\epsilon_i$ and $\sum_{i=1}^{n}\epsilon_i=0$, we conclude that $p_{\Gamma}$ is $T$-invariant if and only if 
$c_{\Gamma}(i)=c_{\Gamma}(j)$ for all $1\leq i, j \leq n$. This proves the lemma.
\end{proof}
    
    \subsection{$Spin_{2n+1}$-standard Young tableau}\label{young}

In this subsection we recall some basic facts about standard Young tableau for $G$, where $G=Spin_{2n+1}$. (see the Appendix in \cite{LP}).     

 Let $\lambda=\Sigma_{i=1}^n a_i\varpi_i$, $a_i \in \mathbb{Z}^{+}$ be a dominant weight. Define $p_i=\Sigma_{j=i}^{n-1}2a_j+a_n$, for $1 \leq i \leq n$.

To $\lambda$ we associate a Young diagram (denoted by $\Gamma$) of shape $p(\lambda)=(p_1,p_2,\ldots)$ with $p_1 \geq p_2 \geq \ldots$ consists of $p_1$ boxes in the first column, $p_2$ in the second column etc.
  
 Let $r=(i_1,\ldots, i_t)$ be a row of length $ t \leq n$ with entries $i_j \leq 2n$. For $i= 1,\ldots, n$ denote by $s_i(r)$ the row defined as follows:
         
       If $i<n$ and $i+1$ and $2n+1-i$ are entries of the row $r$, then $s_i(r)$ is the row obtained from $r$ by replacing the entry $i + 1$ by $i$ and the entry $2n + 1 - i$ by $2n - i$. Else we set $s_i(r) := r$. If $i = n$ and $n + 1$ is an entry of the row $r$, then denote by $s_i(r)$ the row obtained from $r$ by replacing the entry $n + 1$ by $n$. Else we set $s_n(r):=r$.

 We say that a pair of rows $(r,r^\prime)$ are admissible if $r=r^\prime$ or there exists a sequence of different rows $(r_0, r_1, \ldots, r_l)$ such that $r_0=r$, $r_l=r^\prime$ and $s_{i_k}(r_{k-1})=r_k$ for $k=1, 2, \ldots, l$ for some integers $i_1, \ldots, i_l \in \{1, 2, \ldots n\}$.
            
            	A Young diagram $\Gamma$ of shape $p$ is said to be a Young tableau (also denoted by $\Gamma$) if the diagram is filled with positive integers such that
            
   \noindent    1. the entries are less than or equal to $2n$,\\ 
     2. $i$ and $2n+1-i$ do not occur in the same row, for all $1 \leq i \leq n$ and\\
       3. For all $i=1, \ldots, \bar{p_1}$, the pair of rows $(r_{2i-1},r_{2i})$ are admissible, where $\bar{p_1}=\frac{p_1-a_n}{2}$.
            
            The Young tableau is said to be standard if it is strictly increasing in the row and non-decreasing in column. If $i$ is a positive integer and $\Gamma$ is a given Young tableau then we denote by $c_\Gamma(i)$, the number of boxes of $\Gamma$ containing the integer $i$. We define the weight of the young tableau $\Gamma$ as
            
            $$wt(\Gamma):=\frac{1}{2}((c_\Gamma(1)-c_\Gamma(2n))\epsilon_1+\cdots+(c_\Gamma(n)-c_\Gamma(n+1))\epsilon_{n}).$$
            
            Let $P_{\lambda}$ be the parabolic subgroup of $G$ associated to $\lambda$. Then the $T$-eigenvectors of $H^{0}(G/P_{\lambda}, \mathcal L_{\lambda})$ are denoted by $p_{\Gamma}$ which are indexed by the Young tableau $\Gamma$ of shape $p(\lambda)$. We say that $p_{\Gamma}$ is a standard monomial if $\Gamma$ is standard.

              \begin{lemma}\label{zeroweight}
 A monomial $p_{\Gamma}\in H^0(G/P_{\lambda}, \mathcal L_{\lambda})$ is $T$-invariant if and only if $c_{\Gamma}(t)=c_{\Gamma}(2n+1-t)$, for all $ 1 \leq t \leq 2n$.
\end{lemma}

\begin{proof}
A monomial $p_{\Gamma}\in H^0(G/P_{\lambda}, \mathcal L_{\lambda})$ is $T$-invariant if and only if the weight of $\Gamma$ is zero. Recall that weight of a Young tableau $\Gamma$ is given by  $\frac{1}{2}\sum_{j=1}^n(c_{\Gamma}(j)-c_{\Gamma}(2n+1-j))\epsilon_j$. Thus, $p_{\Gamma}$ is $T$-invariant if and only if $c_{\Gamma}(t)=c_{\Gamma}(2n+1-t)$, for all $1 \leq t \leq 2n$.
\end{proof}

The main theorem of the Standard Monomial Theory for any classical group is the following (see \cite{lakshmibai-seshadri}, \cite{LP}): 
         
   \begin{theorem}
              Let $G$ be a simple, simply connected algebraic group and $P_{\lambda}$ be the parabolic subgroup of $G$ associated to a dominant weight $\lambda$. Then the standard monomials $p_{\Gamma}$ form a basis of $H^0(G/P_{\lambda}, \mathcal{L}_{\lambda}^{\otimes m})$ as a vector space, where $\Gamma$ is a standard Young tableau associated to the weight $m\lambda$.  
               \end{theorem}

\section{Some preliminaries of graph theory}\label{graphtheory}

    We follow \cite{AK} for the preliminary definitions in graph theory.
      
	    Let $\mathcal{G}$ be a graph which is represented by the pair $(V(\mathcal{G}),E(\mathcal{G}))$, where $V(\mathcal{G})$ denotes the set of vertices and $E(\mathcal{G})$ denotes the set of edges respectively. A graph having loops and multiple edges is called a \textbf{general graph}. A graph having no loops but having multiple edges is called a \textbf{multigraph}. A graph without loops and at most one edge between any two vertices is called a \textbf{simple graph}. \textbf{Degree} of a vertex in a graph is the number of edges connected to the vertex with loops counted twice. A graph $\mathcal{G}$ is called \textbf{$k$-regular} if each vertex of $V(\mathcal{G})$ is of degree $k$. 
%	    A graph that may have loops and multiple edges is called a general graph.
	    
	    A spanning subgraph of $\mathcal{G}$ is a subgraph of $\mathcal{G}$ which contains every vertex of $\mathcal{G}$. For a positive integer $k$, a \textbf{$k$-factor} of $\mathcal{G}$ is a spanning subgraph of $\mathcal{G}$ that is $k$-regular. A graph $\mathcal{G}$ is said to be \textbf{$k$-factorable} if it has a $k$-factor.
	    
	   A \textbf{walk} in a graph is defined as a sequence of alternating vertices and edges $v_0,e_1,v_1,e_2,\\ \ldots, v_{k-1},e_k,v_k$, where $e_i = (v_{i-1},v_i)$ is the edge between $v_{i-1}$ and $v_i$. The length of this walk is $k$. A walk that passes through every one of its vertices exactly once is called a \textbf{path}. Thus, by an even length path we mean $k$ is even and by an odd length path we mean $k$ is odd. A \textbf{cycle} is a closed path i.e. initial and terminal vertices of the path are same.
	    
	     Let $\mathcal{G}_1$ and $\mathcal{G}_2$ be two graphs where $V(\mathcal{G}_1)$ is same as $V(\mathcal{G}_2)$. Then we define $\mathcal{G}_1 \circ \mathcal{G}_2 = (V(\mathcal{G}_1), E(\mathcal{G}_1) \circ E(\mathcal{G}_2))$, where $E(\mathcal{G}_1) \circ E(\mathcal{G}_2) = \{e ~|~ e \in E(\mathcal{G}_1) \mbox{~or~} e \in E(\mathcal{G}_2)\}$ and $\mathcal{G}_1 \backslash \mathcal{G}_2=(V(\mathcal{G}_1), E(\mathcal{G}_1) \backslash E(\mathcal{G}_2))$, where $E(\mathcal{G}_1) \backslash E(\mathcal{G}_2) = \{e ~|~ e \in E(\mathcal{G}_1) \mbox{~and~} e \notin E(\mathcal{G}_2)\}$.
	    
	    We recall the following two results which will be used in the proof of the main theorem.
	     
	      \begin{theorem}[Petersen's $2$-factor theorem]\cite[Theorem $3.1$, pg.~70]{AK}\label{petersen} For every integer $r \geq 1$, every $2r$-regular general graph is $2$-factorable. More generally, for every integer $k$, $1 \leq k \leq r$, every $2r$-regular general graph has a $2k$-regular factor.
	      \end{theorem}
	      
	      \begin{theorem}\cite[Theorem $2.2$, pg.~18]{AK}\label{1factor}
	Every regular bipartite multigraph is $1$-factorable, in particular, it has a $1$-factor. 
	\end{theorem}
	
%	\textcolor{red}{In Theorem 5.1, by a graph we mean a multigraph and in Section 6, by a graph we mean a general graph.}
	
	In Section 6., our definition of `degree of a vertex' differs from `degree of a vertex' in \cite{AK}. The difference is because of the number of degrees contributed by a loop - in our case, a loop is counted once, however in \cite{AK}, it is counted twice. Since we will be using the results of \cite{AK} directly, so we make the following remark. 
	
	\begin{remark}
\label{rdeg}
{\rm In \cite{AK}, a general graph means a graph with multiple edges and loops where one loop contributes degree $2$ to a vertex incident to it. In our case, one loop contributes degree $1$ to the vertex incident to it. Consider a graph $\mathcal{G}$ with the vertex set $\{v_i : 1 \leq i \leq n\}$, with degree defined as in our case. If $\mathcal{G}$ has even number of loops then the same number of loops at vertex $v_i$ and vertex $v_j$ can be paired up and joined together to get edges between $v_i$ and $v_j$. Doing this procedure will result in even number of loops remaining at a vertex. Now, any two loops at this vertex can be joined together to get a new loop. Now this loop contributes degree $2$ to the vertex. This will result in a graph in \cite{AK}, without changing the degree of any vertices.
For example,
 \[\begin{tikzpicture}[scale=3, every loop/.style={}]
 \vertex[fill] (1) at (0, 1)  [label=above: $v_1$] {}; 
 \vertex [fill](2) at (.5,1.5) [label=left: $v_2$] {};
\vertex[fill] (3) at (1,.3) [label=below: $v_3$]{};
 \vertex[fill] (4) at (.7,0) [label=right: $v_4$] {};   
   \vertex[fill] (5) at (1, 1) [label=above: $v_5$][label=right: 4 \hspace{2.2cm} \mbox{$=$}]  {};  
    \vertex [fill](6) at (1.5,.5) [label=left: $v_6$][label=below: 10] {};  
       \vertex[fill] (7) at (3, 1)  [label=above: $v_1$] {}; 
 \vertex [fill](8) at (3.5,1.5) [label=left: $v_2$] {};
\vertex[fill] (9) at (4,.3) [label=below: $v_3$]{};
 \vertex[fill] (10) at (3.7,0) [label=right: $v_4$] {};   
   \vertex[fill] (11) at (4, 1) [label=above: $v_5$]  {};  
    \vertex [fill](12) at (4.5,.5) [label=left: $v_6$][label=below: 3] {};     
\draw (1.5,.5)  to[in=-50,out=-130,loop] (1.5,.5);
 \draw (1,1)  to[in=30,out=-70,loop] (1,1);
 \draw (4.5,.5)[dashed]  to[in=-50,out=-130,loop] (4.5,.5);
% %\draw (1,1) -- (1.5,.5)[dashed];
   \path
(7) edge node [below] {2} (8)
       (1) edge node [below] {2} (2)
      (1) edge node [below] {3} (4)
       (2) edge node [right] {7} (4)
       (2) edge node [right] {1} (5)
       (1) edge node [below] {5} (3)
      (3) edge node [right] {5} (5)
       (7) edge node [below] {2} (8)
       (7) edge node [below] {5} (9)
       (7) edge node [below] {3} (10)
       (8) edge node [right] {7} (10)
       (8) edge node [right] {1} (11)
       (9) edge node [right] {5} (11)
       (11) edge[dashed] node [right] {4}(12);  
\end{tikzpicture}\]
\hspace{2cm}{Figure 1}
\hspace{7cm}{Figure 2}}

\vspace{0.2cm}

 Here the labels on the edges are their multiplicities. In Figure 1, there are four loops at the vertex $v_5$ and ten loops at the vertex $v_6$. Now we pair up four loops at vertex $v_5$ with four loops at vertex $v_6$ and this results in four edges between vertex $v_5$ and vertex $v_6$ (the dotted lines in Figure 2). After doing this the remaining number of loops at vertex $v_6$ is six. So we pair up two loops together to get a new loop and so this results in three loops at the vertex $v_6$ (the dotted loops in Figure 2), each of which contributes degree $v_2$ to the vertex $v_6$.
 
\end{remark}

 Note that the notion of a $k$-factor is preserved in the modification of the graph respect to the
different notions of degree.
	
	\section{Isomorphic Torus GIT quotients}\label{isomorphism}
Let $G=SL_n(\mathbb C)$ and $T$ be a maximal torus of $G$. Let $\mathcal L_{\omega_r}$ and $\mathcal L_{\omega_{n-r}}$ be the line bundles associated to the fundamental weights $\omega_r$ and $\omega_{n-r}$ respectively. The projective varieties $G_{r,n}$ and $G_{n-r,n}$ are isomorphic. In the following proposition we show that their torus quotients are also isomorphic. 
\begin{proposition}\label{propsan} The GIT quotients $T \backslash \backslash (G_{r,n}) (\mathcal L_{n\omega_r})$ and $T \backslash \backslash (G_{n-r,n}) (\mathcal L_{n\omega_{n-r}})$ are isomorphic. 
\end{proposition}

\begin{proof} Note that $n\varpi_r$ and $n \varpi_{n-r}$ are in the root lattice $Q$. So by \cite[Theorem 3.10]{KS} the line bundle $\mathcal L_{n\varpi_r}$ (resp. $\mathcal L_{n\varpi_{n-r}}$) descends to the quotient $T \backslash \backslash (G_{r,n})_T^{ss} (\mathcal L_{\varpi_r})$ (resp. $T \backslash \backslash (G_{n-r,n})_T^{ss} (\mathcal L_{\varpi_{n-r}})$). 

Let $P_r$ and $P_{n-r}$ be the maximal parabolic subgroups of $G$ corresponding to the simple roots $\alpha_r$ and $\alpha_{n-r}$ respectively.  Let $\mathcal P_r$ (resp. $\mathcal P_{n-r}$) denote the conjugacy class of $P_r$ (resp. $P_{n-r}$) with respect to the conjugation action of $G$. Then there is an $G$-equivariant isomorphism between $G(r,n)$ (resp. $G(n-r,n)$) and the variety $\mathcal P_r$ (resp. $\mathcal P_{n-r}$). 

There exists an outer automorphism $\phi: G \rightarrow G$ that sends $P_r$ to $P_{n-r}$. Note that the outer automorphism comes from the non-trivial diagram automorphism of the Dynkin diagram of $G$. Hence the induced map $\phi_r: \mathcal P_r \rightarrow \mathcal P_{n-r}$, $H \mapsto \phi(H)$ is an isomorphism. This map $\phi_r$ is not $G$-equivariant but the actions of $G$ on $\mathcal P_r$ and $\mathcal P_{n-r}$ are intertwined by $\phi$. That is $\phi(gHg^{-1})=\phi(g)\phi(H)\phi(g)^{-1}$.

Let $T'=\phi(T)$ and let \[q: (\mathcal P_r)_T^{ss}(\mathcal L_{n\varpi_r}) \rightarrow T \backslash \backslash (\mathcal P_r)_T^{ss}(\mathcal L_{n\varpi_r}) \,\, \mbox{and} \,\, q': (\mathcal P_{n-r})_{T'}^{ss}(\mathcal L_{n\varpi_{n-r}}) \rightarrow T' \backslash \backslash (\mathcal P_{n-r})_{T'}^{ss}(\mathcal L_{n\varpi_{n-r}})\] be the quotient morphisms. Since $\phi_r^*(\mathcal L_{n\varpi_{n-r}})=\mathcal L_{n\varpi_r}$ the map $\phi_r$ restricts to an isomorphism (we still call it $\phi_r$) \[(\mathcal P_r)_T^{ss}(\mathcal L_{n\varpi_r}) \rightarrow (\mathcal P_{n-r})_{T'}^{ss}(\mathcal L_{n\varpi_{n-r}}). \] 

Then the map $q' \circ \phi_r:(\mathcal P_r)_T^{ss}(\mathcal L_{n\varpi_r}) \rightarrow T' \backslash \backslash(\mathcal P_{n-r})_{T'}^{ss}(\mathcal L_{n\varpi_{n-r}})$ is a morphism.

 Let $U_T$ be the smallest closed subvariety of $(\mathcal P_r)_T^{ss}(\mathcal L_{n\varpi_r}) \times (\mathcal P_r)_T^{ss}(\mathcal L_{n\varpi_r})$ containing the image of the map $T \times (\mathcal P_r)_T^{ss}(\mathcal L_{n\varpi_r}) \rightarrow (\mathcal P_r)_T^{ss}(\mathcal L_{n\varpi_r}) \times (\mathcal P_r)_T^{ss}(\mathcal L_{n\varpi_r})$ defined by $(t, Q) \mapsto (tQt^{-1}, Q)$ and let $R_T$ be the smallest closed subvariety of $(\mathcal P_r)_T^{ss}(\mathcal L_{n\varpi_r}) \times (\mathcal P_r)_T^{ss}(\mathcal L_{n\varpi_r})$ containing the image of the map  $U_T \times U_T \rightarrow (\mathcal P_r)_T^{ss}(\mathcal L_{n\varpi_r}) \times (\mathcal P_r)_T^{ss}(\mathcal L_{n\varpi_r})$ defined by $((P,Q),(Q,Q')) \mapsto (P,Q')$. Similarly $R_{T'}$ can be defined for the action of $T'$ on $(\mathcal P_{n-r})_{T'}^{ss}\\(\mathcal L_{n\varpi_{n-r}})$.

The product isomorphism $(\phi_r,\phi_r): (\mathcal P_r)_T^{ss}(\mathcal L_{n\varpi_r}) \times (\mathcal P_r)_T^{ss}(\mathcal L_{n\varpi_r}) \rightarrow (\mathcal P_{n-r})_{T'}^{ss}(\mathcal L_{n\varpi_{n-r}}) \times (\mathcal P_{n-r})_{T'}^{ss}(\mathcal L_{n\varpi_{n-r}})$ maps the subvariety $R_T$ isomorphically to the subvariety $R_{T'}$. So the morphism 
$q' \circ \phi_r$ is $T$-invariant. So there exists a unique map $\psi: T \backslash \backslash (\mathcal P_r)_T^{ss}(\mathcal L_{n\varpi_r}) \rightarrow T' \backslash \backslash (\mathcal P_{n-r})_{T'}^{ss}(\mathcal L_{n\varpi_{n-r}})$ such that $\psi \circ q =q' \circ \phi_r$ and it follows that $\psi$ is an isomorphism. 

Since $T$ and $T'$ are conjugate. There exists $g \in G$ such that the conjugation $c_g:G \rightarrow G$, $h \mapsto ghg^{-1}$ restricts to an isomorphism from $T$ to $T'$. Let $r_g: \mathcal P_{n-r} \rightarrow \mathcal P_{n-r}$ be the associated right translation. Then $r_g^*(\mathcal L_{n\varpi_{n-r}})=\mathcal L_{n\varpi_{n-r}}$ and $r_g$ maps $(\mathcal P_{n-r})_T^{ss}(\mathcal L_{n\varpi_{n-r}})$ isomorphically to $(\mathcal P_{n-r})_{T'}^{ss}(\mathcal L_{n\varpi_{n-r}})$. By using the similar arguments as above we see that the quotients $T \backslash \backslash (\mathcal P_{n-r})_T^{ss}(\mathcal L_{n\varpi_{n-r}})$ and $T' \backslash \backslash (\mathcal P_{n-r})_{T'}^{ss}(\mathcal L_{n\varpi_{n-r}})$ are isomorphic to each other. Thus we conclude that the GIT quotients $T \backslash \backslash (G_{r,n}) (\mathcal L_{n\varpi_r})$ and $T \backslash \backslash (G_{n-r,n}) (\mathcal L_{n\varpi_{n-r}})$ are isomorphic.
\end{proof}

\section{Projective normality of the torus quotient of Grassmannians}\label{grassmann}

 For the fundamental weight $\varpi_r$, $n\varpi_r \in Q$. So the line bundle $\mathcal L_{n\varpi_r}$ descends to the quotient $T\backslash\backslash(G_{r,n})_T^{ss}(\mathcal L_{n\varpi_r})$ (see \cite{KS}). In this section we prove that the quotient $T\backslash\backslash(G_{2,n})_T^{ss}(\mathcal L_{n\varpi_2})$ is projectively normal with respect to the descent of $\mathcal{L}_{n\varpi_2}$ using standard monomial theory and some graph theoretic techniques and we give a degree bound of the generators of the homogeneous coordinate ring of $T\backslash\backslash(G_{3,6})_T^{ss}(\mathcal L_{2\varpi_3})$. 

\begin{theorem}\label{Theo1} The GIT quotient $T\backslash\backslash(G_{r,n})_T^{ss}(\mathcal L_{n\varpi_r})$ is projectively normal with respect to the descent of $\mathcal L_{n\varpi_r}$ 
if $r=1,2,n-2,n-1$. 
\end{theorem}
	   	   
\begin{proof} 

For $r=1$, $G_{r,n} \cong \mathbb P^{n-1}$ and hence the quotient $T\backslash\backslash\mathbb P^{n-1}(\mathcal O(n))$ is projectively normal. 

Let $r=2$. We have \[ T\backslash\backslash(G_{2,n})_T^{ss}(\mathcal L_{n\varpi_2}) = Proj({\oplus_{k \in \mathbb{Z}_{\geq 0}}}H^0(G_{2,n},\mathcal{L}^{\otimes k}_{n\varpi_2})^T)=Proj(\oplus_{k\in \mathbb{Z}_{\geq0}}R_k),\]where $R_k :=H^0(G_{2,n},\mathcal{L}^{\otimes k}_{n\varpi_2})^T$. Let $R :=\oplus_{k\in \mathbb{Z}_{\geq0}}R_k$. The $\mathbb{C}$-algebra $R$ is normal since ${\oplus_{k \in \mathbb{Z}_{\geq 0}}}H^0(G_{2,n},\mathcal{L}^{\otimes k}_{n\varpi_2})$ is normal. Hence it is enough to prove that $R$ is generated by $R_1$ as a $\mathbb{C}$-algebra.
	   
As a vector space the $T$-invariant standard monomials in Pl$\ddot{\mbox{u}}$cker coordinates of the form of $\prod_{i<j}p_{ij}^{m_{ij}}$ form a $\mathbb{C}$-basis of $R_k$, where $1 \leq i,j \leq n$. Note that since $\prod_{i<j}p_{ij}^{m_{ij}} \in R_k$ we have $\sum_{j>i} m_{i,j} + \sum_{j<i} m_{j,i}= 2k$ for all $1 \leq i \leq n$ and $\sum_{1 \leq i < j \leq n}m_{ij}=nk$. 
	   
Given a standard monomial $M=\prod_{i<j}p_{ij}^{m_{ij}}$ in Pl$\ddot{\mbox{u}}$cker coordinates we associate a graph as follows. For each $1 \leq i \leq n$ we associate a vertex $v_i$ and for each $p_{ij}$ appearing in $M$ we associate an edge joining the vertex $v_i$ to the vertex $v_j$. Similarly using the reverse process, from every graph, we can associate a monomial in Pl$\ddot{\mbox{u}}$cker coordinates. If moreover $M$ is $T$-invariant then each of the indices $1 \leq i \leq n$ appears exactly $2k$ times in the monomial $M$. So each vertex in the graph is connected to exactly $2k$ number of edges. Hence it is a $2k$-regular graph.

Using Petersen's $2$-factor theorem this graph can be decomposed into $k$ line-disjoint $2$-factors. Each $2$-factor sub-graph is associated to a standard monomial of degree $n$ in Pl$\ddot{\mbox{u}}$cker coordinates, where each integer occurs exactly $2$ times. This associated monomial is standard because the original monomial was standard. The standard monomials associated to the $2$-factor sub-graphs lie in $R_1$. So by induction we conclude that each standard monomial in $R_k$ can be written as a product of $k$ standard monomials in $R_1$. So $R$ is generated by $R_1$ as an algebra and hence the GIT quotient $T\backslash\backslash(G_{2,n})_T^{ss}(\mathcal L_{n\varpi_2})$ is projectively normal with respect to the descent of the line bundle $\mathcal{L}_{n\varpi_2}$.

For $r=n-2$ and $n-1$, the proof follows from \cref{propsan}.

\end{proof}

\begin{corollary}\label{cor1}
 The GIT quotient of a Schubert variety and a Richardson variety in $G_{2,n}$ by a maximal torus $T$ of $SL_n$ is projectively normal with respect to the descent of the line bundle $\mathcal L_{n\varpi_2}$.
 \end{corollary}
 \begin{proof}
Let $X_w$ be a Schubert variety in $G_{2,n}$, $w \in W^{P_{\alpha_2}}$. Since $T$ is linearly reductive, the restriction map $\phi: H^0(G_{2,n}, \mathcal{L}_{n\varpi_2}^{\otimes k})^T \to H^0(X_w, \mathcal{L}_{n\varpi_2}^{\otimes k})^T$ such that $f \mapsto f|_{X_w}$ is surjective. So by \cref{Theo1}, $H^0(X_w, \mathcal{L}_{n\varpi_2}^{\otimes k})^T$ is generated by $H^0(X_w, \mathcal{L}_{n\varpi_2})^T$. Since $(X_w)^{ss}_T(\mathcal{L}_{n\varpi_2})$ is normal so $T \backslash \backslash (X_w)^{ss}_T(\mathcal{L}_{n\varpi_2})$ is projectively normal.
 
 Let $X_w^v$ be a Richardson variety in $G_{2,n}$, $v, w \in W^{P_{\alpha_2}}$. By \cite[Proposition~1]{BL}, the map $H^0(X_w, \mathcal{L}_{n\varpi_2}^{\otimes k}) \to H^0(X_w^v, \mathcal{L}_{n\varpi_2}^{\otimes k})$ is surjective. Since $T$ is linearly reductive, the map $\phi: H^0(X_w, \mathcal{L}_{n\varpi_2}^{\otimes k})^T \to H^0(X_w^v, \mathcal{L}_{n\varpi_2}^{\otimes k})^T$ surjective. 
 Since the quotient $T\backslash\backslash (X_w)^{ss}_T(\mathcal{L}_{n\varpi_2})$ is projectively normal and  $(X^v_w)^{ss}_T(\mathcal{L}_{n\varpi_2})$ is normal, the quotient $T \backslash \backslash (X_w^v)^{ss}_T(\mathcal{L}_{n\varpi_2})$ is projectively normal.
 \end{proof}
 
 For $r \geq 3$, the combinatorics of the standard monomials in $\oplus_{k \in \mathbb{Z}_{\geq 0}}(H^0(G_{r,n}, \mathcal{L}_r^{\otimes k})^T)$ is complicated. So we restrict our case to $n=6$. Again $\mathcal L_{\frac{6}{gcd(6,r)}\varpi_r}$ is the smallest line bundle on $G_{r,6}$ which descends to the quotient $T\backslash\backslash(G_{r,6})_T^{ss}(\mathcal L_{\frac{6}{gcd(6,r)}\varpi_r})$. 
 
 For $r =1, 2, 4$ and $5$ the quotient $T\backslash\backslash(G_{r,6})_T^{ss}(\mathcal L_{\frac{6}{gcd(6,r)}\varpi_r})$ is projectively normal with respect to the descent of the line bundle $\mathcal L_{\frac{6}{gcd(6,r)}\varpi_r}$. For $r=1$, $G_{1,6} \cong \mathbb P^{5}$ and hence the quotient $T\backslash\backslash(\mathbb P^{5})^{ss}_T(\mathcal O(6))$ is projectively normal.
 For $r=2$, $T \backslash\backslash (G_{2,6})^{ss}_T(\mathcal{L}_{3\varpi_2})$ is projectively normal as proved in \cite{HMSV}.
 For $r =4$ and $5$ the quotient $T\backslash\backslash(G_{r,6})_T^{ss}(\mathcal L_{\frac{6}{gcd(6,r)}\varpi_r})$ is projectively normal by \cref{propsan}.

 In the following theorem we give a degree bound of the generators of the homogeneous coordinate ring of the quotient  $T\backslash\backslash(G_{3,6})_T^{ss}(\mathcal L_{2\varpi_3})$.

\begin{theorem}\label{Theo2}
The homogeneous coordinate ring of the quotient $T\backslash\backslash(G_{3,6})_T^{ss}(\mathcal L_{2\varpi_3})$ is generated by elements of degree at most $2$.
\end{theorem}

\begin{proof}
We have \[ T\backslash\backslash(G_{3,6})_T^{ss}(\mathcal L_{2\varpi_3}) = Proj({\oplus_{k \in \mathbb{Z}_{\geq 0}}}H^0(G_{3,6},\mathcal{L}^{\otimes k}_{2\varpi_3})^T)=Proj(\oplus_{k\in \mathbb{Z}_{\geq0}}R_k),\]
where $R_k = H^0(G_{3,6}, \mathcal{L}_{2\omega_3}^{\otimes k})^T$. Let $M$ be a standard monomial in Pl$\ddot{\mbox{u}}$cker coordinates in $R_k$. Then $M$ is associated to a $2k \times 3$ tableau having each of the integers from $1$ to $6$ appearing exactly $k$ times with strictly increasing rows and non-decreasing columns. Let $\textrm{Row}_{i}$ denote the $i$th row of the tableau and $\textrm{Col}_{j}$ denote the $j$th column of the tableau, where $1 \leq i \leq 2k$ and $1 \leq j \leq 3$. Let $E_{i,j}$ be the $(i,j)$-th entry of the tableau and $N_{t,j} = \#\{ i| E_{i,j}=t\}$. Clearly, 
\begin{equation}\label{colsum} 
\sum_{t} N_{t,j} = 2k \quad \textrm{and} \quad  \sum_{j} N_{t,j} = k.
\end{equation}

Note that $E_{i,1} = 1$ for all $1 \leq i \leq k$ and $E_{i,3} = 6$ for all $ k+1 \leq i \leq 2k$.

 If $E_{1,2} = 4$ then $N_{t,1} = k$ for $1 \leq t \leq 3$, a contradiction. Similarly $E_{1,2}$ cannot be $5$. So, $\textrm{Row}_1$ can be one of the elements from the set $\{(1,2,3), (1,2,4), (1,2,5), (1,3,4), (1,3,5)\}$.

If $\textrm{Row}_1 = (1,3,5)$ then we have $N_{2,1}=k$ 
and $E_{i,1} = 2$ for all $k+1 \leq i \leq 2k$. In particular, we have $E_{2k,1}=2$. Since $E_{1,3}=5$, we have $N_{4,2}=k$ and $N_{3,2}=k$. So we have $E_{2k,2}=4$. Hence we conclude that, $\textrm{Row}_{2k} = (2,4,6)$. Then $p_{135}p_{246} \in R_1$ and divides $M$. So by induction we are done. 

If $\textrm{Row}_1 = (1,3,4)$ then we have $E_{2k,1} = 2$. Since $E_{1,3}=4$ we have $N_{5,2} \geq 1$. So $E_{2k,2}=5$. Hence we conclude that, $\textrm{Row}_{2k} = (2,5,6)$. Then  $p_{134}p_{256} \in R_1$ and is a factor of $M$. 

 If $\textrm{Row}_1 = (1,2,5)$ then $N_{5,3} = k$ and $E_{2k,2}=4$. Since $E_{1,2}=2$ we have $N_{3,1} \geq 1$ and so $E_{2k,1} = 3$. So $\textrm{Row}_{2k} = (3,4,6)$. Then  $p_{125}p_{346} \in R_1$ and is a factor of $M$. 

We are now left with two cases, either $\textrm{Row}_1 = (1,2,3)$ or $\textrm{Row}_1 = (1,2,4)$
 
\textbf{Case - 1} {$\textrm{Row}_1 = (1,2,4)$}

Since $E_{1,3} = 4$ we have $N_{5,3} < k$. Since $N_{5,1}=0$ it follows that $N_{5,2} \geq 1$ and hence, $E_{2k,2}=5$. If $N_{4,1} = 0$ then $N_{3,1} \geq 1$. It follows that $\textrm{Row}_{2k} = (3,5,6)$. So the monomial $p_{124}p_{356} \in R_1$ and is a factor of $M$. If $N_{4,1} \geq 1$ then $E_{2k,1}=4$ and hence  $\textrm{Row}_{2k} = (4,5,6)$. We claim that $\textrm{Row}_{k} = (1,3,5)$. 

 (a) If $E_{k,2}=2$ then we have $E_{i,2} =2$ for all $1 \leq i \leq k$. Since $E_{1,3}=4$ we have $N_{3,3} =0$. Since $\textrm{Row}_{2k}=(4,5,6)$ we have $N_{3,1}+N_{3,2} < k$, a contradiction.
 
 (b) If $E_{k,2}=4$, then $(N_{4,1}+N_{4,2}+N_{4,3})+(N_{5,2}+N_{5,3}) \geq 2k+2$, which is a contradiction.
 
 (c) For a similar reason we cannot have $E_{k,2}=5$.
 
  Hence, $E_{k,2}=3$. 
  
  If $E_{k,3}=4$ then $E_{i,3}=4$ for all $1 \leq i \leq k$. So, $N_{4,1}+N_{4,3} \geq k+1$, a contradiction.\\
  So we conclude that $\textrm{Row}_{k} = (1,3,5)$, the claim is proved.
  
Now we consider the entries $E_{i,2}$, where $1 \leq i \leq k$. Since $E_{1,2} = 2$ and $E_{k,2}=3$ we have $E_{i,2} = 2$ or $3$ for $1 \leq i \leq k$. Let $m_1 = \#\{ i :  E_{i,2}=2, 1 \leq i \leq k\}$ and $m_2 = \#\{ i :  E_{i,2}=3, 1 \leq i \leq k\}$. Then $m_1,m_2\geq 1$ and $m_1+m_2=k$.

\textbf{Subcase - $1$}. $m_1=m_2= \frac{k}{2}$.
 
 (a) If $N_{4,3}= \frac{k}{2}$ then $\textrm{Row}_i = (1,2,4)$ for all $1 \leq i \leq \frac{k}{2}$ and $\textrm{Row}_i= (1,3,5)$ for all $\frac{k}{2}+1 \leq i \leq k$. Then the monomial $M$ is $p_{124}^{\frac{k}{2}}p_{135}^{\frac{k}{2}}p_{236}^qp_{246}^{\frac{k}{2}-q}p_{356}^{\frac{k}{2}-q}p_{456}^q$ with $q \geq 1$. \\
If $q < \frac{k}{2}$ then $M$ has a factor $p_{124}p_{356} \in R_1$.\\
 If $q=\frac{k}{2}$ then the monomial $M$ is $(p_{124}p_{135}p_{236}p_{456})^\frac{k}{2}$. Then $p_{124}p_{135}p_{236}p_{456} \in R_2$ and is a factor of $M$. 
 
  (b) If $N_{4,3} < \frac{k}{2}$ then $N_{5,3} > \frac{k}{2}$ and so $E_{\frac{k}{2},3}=5$. Hence, $\textrm{Row}_{\frac{k}{2}}=(1,2,5)$. Since $N_{5,3} > \frac{k}{2}$ we have $N_{5,2} < \frac{k}{2}$ and since $N_{2,1} = \frac{k}{2}$ we have $E_{i,1}=2$ for all $k+1 \leq i \leq \frac{3k}{2}$ and $2 \notin \textrm{Row}_{\frac{3k}{2}+1}$. Since $N_{5,2} < \frac{k}{2}$ we have $5 \notin E_{\frac{3k}{2}+1,2}$ and hence $\textrm{Row}_{\frac{3k}{2}+1} = (3,4,6)$. Then the monomial $p_{125}p_{346} \in R_1$ and is a factor of $M$. 
  
(c)If $N_{4,3}> \frac{k}{2}$ then $\textrm{Row}_{\frac{k}{2}+1}=(1,3,4)$. Now using a similar argument as (b) we get $p_{134}p_{256} \in R_1$ and is a factor of $M$.
 
\textbf{Subcase - $2$}. Let $m_1 \neq m_2$.
 
 Let $m_1 > m_2$. Note that $m_1  > \frac{k}{2}$.
 
(a) If $N_{4,3} = m_1$ then $\textrm{Row}_{i}=(1,2,4)$ for all $1 \leq i \leq m_1$ and $\textrm{Row}_{i}=(1,3,5)$ for all $m_1+1 \leq i \leq m_2$ i.e. $N_{5,3}=m_2$, $N_{5,2}=m_1$ and $N_{2,1}=m_2$. Hence, $N_{3,2}<2m_2<k$ and it follows that $N_{3,1} > 1$. So $\textrm{Row}_{k+m_2+1} = (3,5,6)$. So the monomial $p_{124}p_{356} \in R_1$ and is a factor of $M$.

(b) If $N_{4,3} > m_1$ then $\textrm{Row}_{m_1+1}=(1,3,4)$ and $N_{5,3} < m_2$. Hence $N_{5,2} > m_1$ and $\textrm{Row}_{k+m_2} = (2,5,6)$. So the monomial $p_{134}p_{256} \in R_1$ and is a factor of $M$.

(c) If $N_{4,3} < m_1$, Now using a similar argument as (b) we get $p_{125}p_{346} \in R_1$ and is a factor of $M$.
 
 The proof for the case $m_1<m_2$ is similar.
 
 \textbf{Case - 2} {$\textrm{Row}_1 = (1,2,3)$}
 
 Similarly as in Case - 1 we see that either $M$ has a factor in $R_1$ or $p_{123}p_{145}p_{246}p_{356}$ divides $M$ and is an element of $R_2$. 
 
 So by induction we conclude that $M$ is generated by the elements of degree at most $2$ and hence  the homogeneous coordinate ring of the quotient $T\backslash\backslash(G_{3,6})_T^{ss}(\mathcal L_{2\varpi_3})$ is generated by elements of degree at most $2$.
 \end{proof}
 
 \section{Torus quotient of partial flag varieties}\label{partialflag}

Let $G =SL_n$ and $\varpi_1$, $\varpi_2$ be the fundamental weights associated to the simple roots $\alpha_1$ and $\alpha_2$ respectively. Let $P = P_{\alpha_1} \cap P_{\alpha_2}$. Since $n(r_1\varpi_1+r_2\varpi_2) \in Q$ for $r_1, r_2 \in \mathbb{N}$, the line bundle $\mathcal{L}_{n(r_1\varpi_1+r_2\varpi_2)}$ descends to the quotient $T \backslash\backslash (G/P)(\mathcal{L}_{n(r_1\varpi_1+r_2\varpi_2)})$ (see \cite{KS}). In this section we prove that the quotient $T \backslash\backslash (G/P)^{ss}_T(\mathcal{L}_{n(r_1\varpi_1+r_2\varpi_2)})$ is projectively normal with respect to the descent of the line bundle $\mathcal{L}_{n(r_1\varpi_1+r_2\varpi_2)}$.

\begin{remark}
\label{2reg}
Any $2$-regular graph is a disjoint union of (i) even cycles, (ii) odd cycles, (iii) even length paths starting with a loop and ending with a loop, (iv) odd length paths starting with a loop and ending with a loop, and (v) vertices with two loops.
\end{remark}

\begin{proof} Recall that in our case a loop contributes degree $1$ to a vertex. It is well known that a $2$-regular connected simple graph is a cycle \cite[pg.~83]{AK}. If the graph is not simple then it may have loops and multiple edges. If it has multiple edges then at least two of the  vertices are connected by two edges, hence, it is a $2$-cycle. If the graph has a loop at a vertex $v$ then either $v$ has another loop around it or it is connected to another vertex $w$ by an edge. In the later case $w$ may have another loop around it or connected to another vertex $u$ by an edge. In the former case the graph is an odd length path starting with a loop and ending with a loop and continuing this process we get either an even length path starting with a loop and ending with a loop or an odd length path starting with a loop and ending with a loop. 
\end{proof}
%\vspace{-.655cm}

%\vspace{-.65cm}

Now we are in a position to state and prove the main theorem of this section.

 \begin{theorem}\label{Theo3}
Let $G=SL_n$ and $P = P_{\alpha_1} \cap P_{\alpha_2}$, $\varpi = r_1\varpi_1+r_2\varpi_2$. The GIT quotient $T\backslash\backslash(G/P)_T^{ss}(\mathcal L_{n\varpi})$ is projectively normal with respect to the descent of the line bundle $\mathcal{L}_{n\varpi}$.
  \end{theorem} 
  
    \begin{proof}
  Note that $T\backslash\backslash(G/P)_T^{ss}(\mathcal L_{n\varpi}) = Proj({\oplus_{k \in \mathbb{Z}_{\geq 0}}}H^0(G/P,\mathcal{L}^{\otimes k}_{n\varpi})^T)=Proj(\oplus_{k\in \mathbb{Z}_{\geq0}}R_k)$, 
where $R_k=H^0(G/P,\mathcal{L}^{\otimes k}_{n\varpi})^T$. The algebra $R=\oplus_{k \in \mathbb{Z}_{\geq 0}}R_k$ is normal. Here we use induction on $k$ to prove that $R$ is generated by $R_1$ as a $\mathbb{C}$-algebra. We set $s = r_1+2r_2$, $l_1 = n(r_1+r_2)$ and $l_2 = nr_2$.

Let $f = \prod_{t=1}^{kl_2}p_{i_tj_t}\prod_{t=kl_2+1}^{kl_1}p_{m_t} \in R_k$ be a standard monomial in the Pl$\ddot{\mbox{u}}$cker coordinates. We associate a graph $\mathcal{G}_f$ corresponding to $f$ as follows:\\
(a) for each integer $1 \leq i\leq n$, associate a vertex $v_i$,\\
(b) for each $p_{ij}$ appearing in $f$, associate an edge between $v_i$ and $v_j$, and\\
(c) for each $p_k$ appearing in $f$, associate a loop at the vertex $v_k$.\\
Similarly, using the reverse process, we can associate a monomial $f_{\mathcal{G}}$ in Pl\"{u}cker coordinates with a graph $\mathcal{G}$.

For $f \in R_k$, the associated graph $\mathcal{G}_f$ has total $k(l_1-l_2)=knr_1$ number of loops. Since $f \in R_k$, it is $T$ invariant and so each of the indices $1 \leq i \leq n$ appears exactly $ks$-times in the monomial $f$. This results in all the vertices of $\mathcal{G}_f$ having the same degree. Thus $\mathcal{G}_f$ is $ks$-regular.

 Here we introduce some operations on the graphs which are induced by the operations on the monomials corresponding to the graphs:

$  \begin{array}{ccc}
   \mathcal{G} & = \mathcal{G}_1+\mathcal{G}_2 & \text{if $f_{\mathcal{G}} = f_{\mathcal{G}_1}+f_{\mathcal{G}_2}$},\\
    \mathcal{G} & = \mathcal{G}_1-\mathcal{G}_2 & \text{if $f_{\mathcal{G}} = f_{\mathcal{G}_1}-f_{\mathcal{G}_2}$},\\
    \mathcal{G} & = \mathcal{G}_1\circ \mathcal{G}_2 & \text{if $f_{\mathcal{G}} = f_{\mathcal{G}_1}.f_{\mathcal{G}_2}$},\\
    \mathcal{G} & = \mathcal{G}_1\backslash\mathcal{G}_2 & \text{if $f_{\mathcal{G}} = f_{\mathcal{G}_1}/f_{\mathcal{G}_2}$},\\
    \end{array} $
    
    where ${f}_\mathcal{G}$ and ${f}_{\mathcal{G}_i}$ are the monomials associated to the graphs $\mathcal{G}$ and $\mathcal{G}_i$ respectively, for $i = 1, 2$.

 We proceed case by case and in each case we first show that $\mathcal{G}_f$ is a linear combination of $ks$-regular graphs and from each summand we get a $s$-factor. 
%This $s$-factor will correspond to an element of $R_1$. 

  \textbf{Case 1: $r_1$ is even.}

 In this case $s$ is even and the number of loops in the graph is even. So by \cref{petersen} and Remark \ref{rdeg}, $\mathcal{G}_f$ has a $s$-factor.

     \textbf{Case 2: $r_1$ is odd.}

In this case $s$ is odd. 
%Here we may assume that the number of loops in $\mathcal{G}_f$ is greater than or equal to $2$,  otherwise the number of loops is zero, in which case the proof follows from \cref{Theo1}.
 We consider two cases, $k$ is even and $k$ is odd.

\underline{\textbf{$k$ is even.}}
 
In this subcase the number of loops in the graph is even. So  by \cref{petersen}, $\mathcal{G}_f$ can be factored into $\frac{ks}{2}$ number of $2$-factors. We make the following claim. 

\textbf{Claim 1:} One of the $2$-factors can be written as a linear combination of $2$-regular graphs such that from each of the summands we can extract a $1$-factor.

 Since $\mathcal{G}_f$ had at least two loops, by Remark \ref{2reg}, one of the $2$-factors also has at least two loops. Denote this particular $2$-factor (with at least two loops) by $\mathcal{G}_{f^{(2)}}$. 
 
Since $\mathcal{G}_{f^{(2)}}$ is a $2$-regular graph, using Remark \ref{2reg}, $\mathcal{G}_{f^{(2)}}$ is a disjoint union of (i) even cycles, (ii) odd cycles, (iii) even length paths starting with a loop and ending with a loop, (iv) odd length paths starting with a loop and ending with a loop, and (v) vertices with two loops.

 Now, to get the graph free of odd cycles we merge two odd cycles together by taking one edge from each and apply Pl\"{u}cker relations on them.  

 In the following example we use the Pl\"{u}cker relation $p_{13}p_{45}=p_{14}p_{35}-p_{15}p_{34}$ on the edges $(v_1,v_3)$ and $(v_4,v_5)$ to merge two odd cycles. 

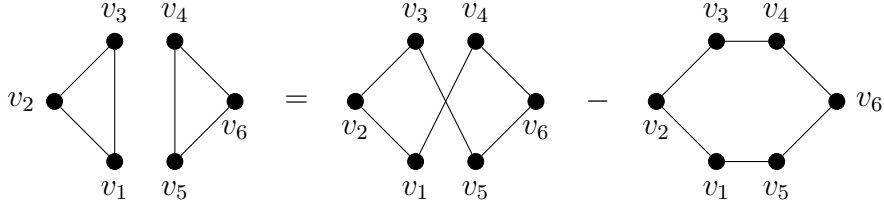
\[\begin{tikzpicture}[scale=4, every loop/.style={}]
%	%% Notice in the first vertex is named (v) for the sake of a later edge,
%	%% and it also has a label to its left that is the math-mode $v$. 
	\vertex[fill] (1) at (.2, 0) [label=below:$v_1$]  {};  
	\vertex [fill](2) at (0,.2) [label=left:$v_2$] {};
	\vertex[fill] (3) at (.2,.4) [label=above:$v_3$] {};
	\vertex[fill] (4) at (.4,.4) [label=above:$v_4$] {};	
	\vertex[fill] (5) at (.4, 0) [label=below:$v_5$] {};  
	\vertex [fill](6) at (.6,.2) [label=below:$v_6$][label=right:\hspace{.4cm}\mbox{=}] {};
	\vertex[fill] (7) at (1.2, 0) [label=below:$v_1$]  {};  
	\vertex [fill](8) at (1,.2) [label=below:$v_2$] {};
	\vertex[fill] (9) at (1.2,.4) [label=above:$v_3$] {};
	\vertex[fill] (10) at (1.4,.4) [label=above:$v_4$] {};	
	\vertex[fill] (11) at (1.6, .2) [label=below:$v_6$] [label=right:\hspace{.4cm}\mbox{$-$}] {};  
	\vertex [fill](12) at (1.4,0) [label=below:$v_5$] {};
	\vertex[fill] (13) at (2.2, 0) [label=below:$v_1$] {};  
	\vertex [fill](14) at (2,.2)  [label=below:$v_2$]  {};
	\vertex[fill] (15) at (2.2,.4)  [label=above:$v_3$]{};
	\vertex[fill] (16) at (2.4,.4) [label=above:$v_4$] {};	
	\vertex[fill] (17) at (2.6,.2) [label=right: $v_6$]  {};  
	\vertex [fill](18) at (2.4,0) [label=below:$v_5$] {};	
%	\node[style={circle,draw=black, very thick,inner sep=5pt}] at (42) at (-1,4) {};
% \draw (.6,.2)  to[in=-50,out=-130,loop] (.6,.2);
%   \draw (1.6,0)  to[in=-50,out=-130,loop] (1.6,0);
%    \draw (2,.2)  to[in=-50,out=-130,loop] (2,.2);
%   \draw (2.8,.4)  to[in=-50,out=-130,loop] (2.8,.4);
%   \draw (3.2,.2)  to[in=-50,out=-130,loop] (3.2,.2);
	\path
%	   % Note that the word "path" here isn't used in the graph-theory sense; the \path command
%	   % is always used prior to the list of edges; here, coincidentally, they do form an actual path.
		(1) edge (2)
		(1) edge (3)
		(2) edge (3)
		(4) edge (5)
		(4) edge (6)
		(5) edge (6)
		(7) edge (8)
		(8) edge (9)
		(9) edge (12)
		(7) edge (10)
		(10) edge (11)
		(11) edge (12)
		(13) edge (14)
		(14) edge (15)
		(15) edge (16)
		(13) edge (18)
		(16) edge (17)
		(17) edge (18)	 ;   % This semicolon ends the \path command.
\end{tikzpicture}\]

  Repeating this process we write $\mathcal{G}_{f^{(2)}} = \sum_{i=1}^pa_i\mathcal{G}_{f^{(2)}_i}$, where $a_i \in \mathbb{Z}$ and each $\mathcal{G}_{f^{(2)}_i}$ is a $2$-regular graph which is a disjoint union of (i) even cycles, (ii) even length paths starting with a loop and ending with a loop, (iii) odd length paths starting with a loop and ending with a loop, (iv) vertices with two loops and (v) \textbf{possibly one odd cycle}. 
  
(a) Suppose $\mathcal{G}_{f^{(2)}_i}$ has no odd cycle. We can extract a $1$-factor from $\mathcal{G}_{f^{(2)}_i}$ in the following ways:\\
If $\mathcal{G}_{f^{(2)}_i}$ has an even cycle as a component it can be factored into two $1$-factors by taking every alternate edge.\\
If $\mathcal{G}_{f^{(2)}_i}$ has an even length path starting with a loop and ending with a loop as a component we pick a loop and every alternate edge to get a $1$-factor. \\
If $\mathcal{G}_{f^{(2)}_i}$ has an odd length path stating with a loop and ending with a loop as a component we pick up the two loops and every alternate edge to get a $1$-factor. \\
If $\mathcal{G}_{f^{(2)}_i}$ has a vertex with two loops as a component we take one loop from it.

(b) Suppose $\mathcal{G}_{f^{(2)}_i}$ has an odd cycle. Since $\mathcal{G}_{f^{(2)}}$ has at least two loops, $\mathcal{G}_{f^{(2)}_i}$ will also have at least two loops. To get the graph free of the odd cycle we choose an edge $(v_i,v_j)$ in the odd cycle, and a loop $(v_k,v_k)$ (w.l.o.g  $\{i,j,k:i<j<k\}$) and apply the Pl\"{u}cker relation $p_{ij}p_k=p_{ik}p_j - p_{jk}p_i$.\\
We may take an odd cycle and one of the components of the following types to apply Pl\"{u}cker relation:\\
a.  vertex with two loops.\\
b. even length path starting with a loop and ending with a loop.\\
c. odd length path starting with a loop and ending with a loop.\\
In the following examples the Pl\"{u}cker relation $p_{13}p_4=p_{14}p_3-p_{34}p_1$ is applied on the edge $(v_1,v_3)$ and the loop $(v_4,v_4)$.\\ 

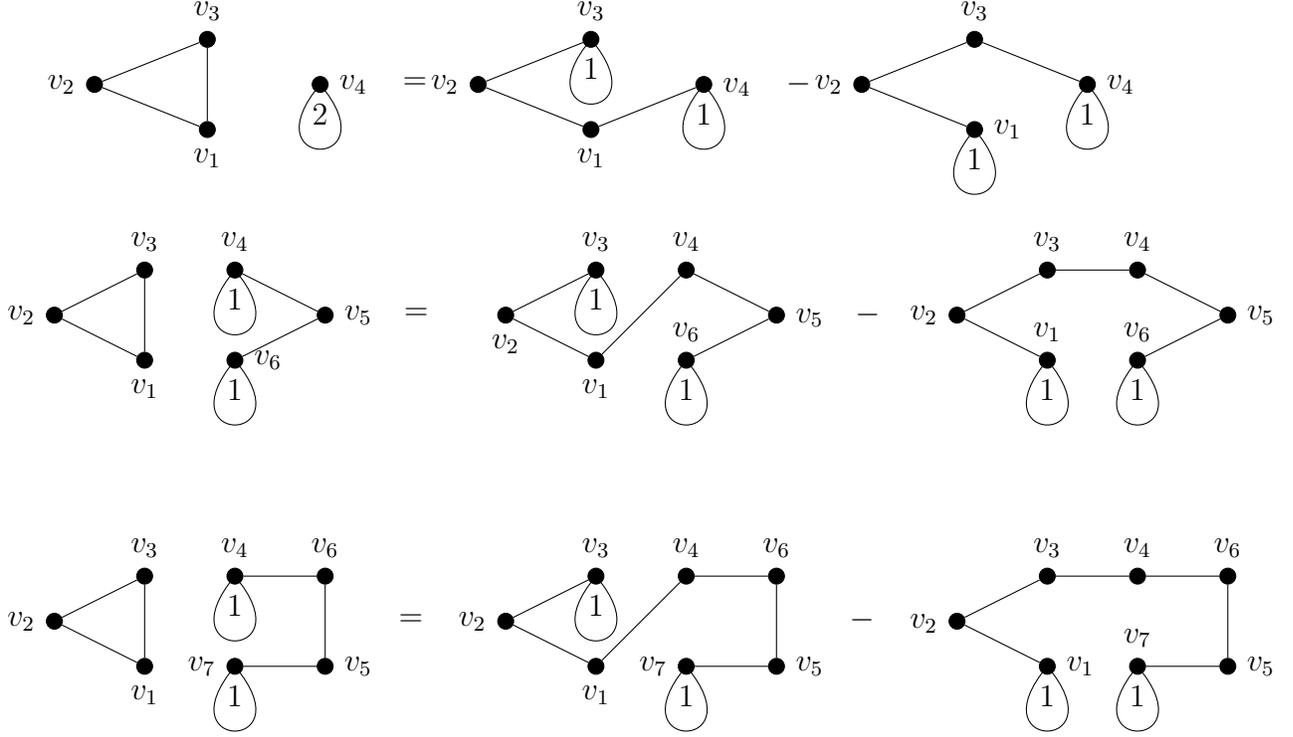
\[\begin{tikzpicture}[scale=3, every loop/.style={}]
	%% Notice in the first vertex is named (v) for the sake of a later edge,
	%% and it also has a label to its left that is the math-mode $v$. 
	\vertex[fill] (1) at (.5, 0)  [label=below: $v_1$] {};  
	\vertex [fill](2) at (0,.2) [label=left: $v_2$] {};
	\vertex[fill] (3) at (.5,.4) [label=above: $v_3$]{};
	\vertex[fill] (4) at (1,.2) [label=right: $v_4$ \hspace{.2cm} \mbox{$=$}] [label=below: 2] {};	
	\vertex[fill] (5) at (2.2, 0) [label=below: $v_1$]  {};  
	\vertex [fill](6) at (1.7,.2) [label=left: $v_2$] {};
	\vertex[fill] (7) at (2.2,.4)[label=above: $v_3$] [label=below: 1]{};
	\vertex[fill] (8) at (2.7,.2)[label=right: $v_4$ \hspace{.2cm} \mbox{$-$}] [label=below: 1] {};
	\vertex[fill] (9) at (3.9, 0) [label=right: $v_1$] [label=below: 1] {};  
	\vertex [fill](10) at (3.4,.2) [label=left: $v_2$] {};
	\vertex[fill] (11) at (3.9,.4) [label=above: $v_3$] {};
	\vertex[fill] (12) at (4.4,.2) [label=right: $v_4$][label=below: 1] {};		
%	\node[style={circle,draw=black, very thick,inner sep=5pt}] at (42) at (-1,4) {};
  \draw (1,.2)  to[in=-50,out=-130,loop] (.6,.2);
   \draw (2.2,.4)  to[in=-50,out=-130,loop] (2.2,.4);
    \draw (2.7,.2)  to[in=-50,out=-130,loop] (2.7,.2);
   \draw (3.9,0)  to[in=-50,out=-130,loop] (3.9,0);
   \draw (4.4,.2)  to[in=-50,out=-130,loop] (3.2,.2);
	\path
	   % Note that the word "path" here isn't used in the graph-theory sense; the \path command
	   % is always used prior to the list of edges; here, coincidentally, they do form an actual path.
		(1) edge (2)
		(1) edge (3)
		(2) edge (3)
		(5) edge (6)
		(6) edge (7)
		(5) edge (8)
		(9) edge (10)
		(10) edge (11)
		(11) edge (12)
	 ;   % This semicolon ends the \path command.
\end{tikzpicture}\]

\vspace{-1cm}
\[\begin{tikzpicture}[scale=3, every loop/.style={}]
	%% Notice in the first vertex is named (v) for the sake of a later edge,
	%% and it also has a label to its left that is the math-mode $v$. 
	\vertex[fill] (1) at (.4, 0) [label=below: $v_1$]  {};  
	\vertex [fill](2) at (0,.2) [label=left: $v_2$] {};
	\vertex[fill] (3) at (.4,.4)[label=above: $v_3$] {};
	\vertex[fill] (4) at (.8,0) [label=right: $v_6$][label=below: 1] {};	
	\vertex[fill] (5) at (1.2,.2) [label=right: $v_5$ \hspace{.15cm} \mbox{$=$}]  {};  
	\vertex [fill](6) at (.8,.4) [label=above: $v_4$][label=below: 1] {};
	\vertex[fill] (7) at (2.4,0) [label=below: $v_1$] {};
	\vertex[fill] (8) at (2,.2) [label=below: $v_2$] {};
	\vertex[fill] (9) at (2.4,.4)[label=above: $v_3$] [label=below: 1]  {};  
	\vertex [fill](10) at (2.8,0) [label=above: $v_6$][label=below: 1] {};
	\vertex[fill] (11) at (3.2,.2) [label=right: $v_5$ \hspace{.15cm} \mbox{$-$}] {};
	\vertex[fill] (12) at (2.8,.4) [label=above: $v_4$] {};
	\vertex[fill] (13) at (4.4, 0) [label=above: $v_1$][label=below: 1]  {};  
	\vertex [fill](14) at (4,.2) [label=left: $v_2$] {};
	\vertex[fill] (15) at (4.4,.4) [label=above: $v_3$]{};
	\vertex[fill] (16) at (4.8,0) [label=above: $v_6$][label=below: 1] {};	
	\vertex[fill] (17) at (5.2,.2) [label=right: $v_5$] {};  
	\vertex [fill](18) at (4.8,.4) [label=above: $v_4$] {};		
%	\node[style={circle,draw=black, very thick,inner sep=5pt}] at (42) at (-1,4) {};
  \draw (.8,0)  to[in=-50,out=-130,loop] (.8,0);
   \draw (.8,.4)  to[in=-50,out=-130,loop] (.8,.4);
    \draw (2.4,.4)  to[in=-50,out=-130,loop] (2.4,.4);
   \draw (2.8,0)  to[in=-50,out=-130,loop] (2.8,0);
   \draw (4.4,0)  to[in=-50,out=-130,loop] (4.4,0);
   \draw (4.8,0)  to[in=-50,out=-130,loop] (4.8,0);
	\path
	   % Note that the word "path" here isn't used in the graph-theory sense; the \path command
	   % is always used prior to the list of edges; here, coincidentally, they do form an actual path.
	   (1) edge (2)
		(1) edge (3)
		(2) edge (3)
		(4) edge (5)
		(5) edge (6)
		(7) edge (8)
		(9) edge (8)
		(10) edge (11)
		(11) edge (12)
		(7) edge (12)
		(13) edge (14)
		(14) edge (15)
		(15) edge (18)
		(17) edge (18)
		(16) edge (17)
	 ;   % This semicolon ends the \path command.
\end{tikzpicture}\]

\[\begin{tikzpicture}[scale=3, every loop/.style={}]
	%% Notice in the first vertex is named (v) for the sake of a later edge,
	%% and it also has a label to its left that is the math-mode $v$. 
	\vertex[fill] (1) at (.4,0) [label=below: $v_1$]  {};  
	\vertex [fill](2) at (0,.2) [label=left: $v_2$] {};
	\vertex[fill] (3) at (.4,.4)[label=above: $v_3$] {};
	\vertex[fill] (4) at (.8,0)[label=left: $v_7$] [label=below: 1] {};	
	\vertex[fill] (5) at (1.2,0)  [label=right: $v_5$] {};  
	\vertex [fill](6) at (.8,.4) [label=above: $v_4$][label=below: 1] {};
	\vertex[fill] (19) at (1.2,.4) [label=above: $v_6$]  {};
	\vertex[fill] (7) at (2.4,0) [label=below: $v_1$]{};
	\vertex[fill] (8) at (2,.2) [label=left: \mbox{$=$} \hspace{.2cm} $v_2$ ] {};
	\vertex[fill] (9) at (2.4,.4) [label=above: $v_3$][label=below: 1]  {};  
	\vertex [fill](10) at (2.8,0) [label=left: $v_7$][label=below: 1] {};
	\vertex[fill] (11) at (3.2,0) [label=right: $v_5$] {};
	\vertex[fill] (20) at (3.2,.4)[label=above: $v_6$] {};
	\vertex[fill] (12) at (2.8,.4) [label=above:$v_4$] {};
	\vertex[fill] (13) at (4.4, 0)[label=right: $v_1$] [label=below: 1]  {};  
	\vertex [fill](14) at (4,.2)[label=left: \mbox{$-$} \hspace{.2cm} $v_2$]  {};
	\vertex[fill] (15) at (4.4,.4)[label=above: $v_3$] {};
	\vertex[fill] (16) at (4.8,0) [label=above: $v_7$][label=below: 1] {};	
	\vertex[fill] (17) at (5.2,0) [label=right: $v_5$] {};  
	\vertex [fill](18) at (4.8,.4) [label=above: $v_4$] {};
	\vertex[fill](21)	at (5.2,.4) [label=above: $v_6$]{};	
%	\node[style={circle,draw=black, very thick,inner sep=5pt}] at () at (-1,4) {};
  \draw (.8,0)  to[in=-50,out=-130,loop] (.8,0);
   \draw (.8,.4)  to[in=-50,out=-130,loop] (.8,.4);
    \draw (2.4,.4)  to[in=-50,out=-130,loop] (2.4,.4);
   \draw (2.8,0)  to[in=-50,out=-130,loop] (2.8,0);
   \draw (4.4,0)  to[in=-50,out=-130,loop] (4.4,0);
   \draw (4.8,0)  to[in=-50,out=-130,loop] (4.8,0);
	\path
	   % Note that the word "path" here isn't used in the graph-theory sense; the \path command
	   % is always used prior to the list of edges; here, coincidentally, they do form an actual path.
	   (1) edge (2)
		(1) edge (3)
		(2) edge (3)
		(4) edge (5)
		(5) edge (19)
		(6) edge (19)
		(7) edge (8)
		(9) edge (8)
		(10) edge (11)
		(11) edge (20)
		(12) edge (20)
		(7) edge (12)
		(13) edge (14)
		(14) edge (15)
		(15) edge (18)
		(17) edge (21)
		(16) edge (17)
		(18) edge (21)
	 ;   % This semicolon ends the \path command.
\end{tikzpicture}\]

  After doing this we write $\mathcal{G}_{f^{(2)}_i} = \sum_{k=1}^{m_i} b_{i_k}\mathcal{G}_{f^{(2)}_{i_k}}$, where $b_{i_k} \in \mathbb{Z}$ and each $\mathcal{G}_{f^{(2)}_{i_k}}$ is a $2$-regular graph which is a disjoint union of (i) even cycles, (ii) even length paths starting with a loop and ending with a loop, (iii) odd length paths starting with a loop and ending with a loop, (iv) vertices with two loops. So, from each of the components of $\mathcal{G}_{f^{(2)}_{i_k}}$ we can extract a $1$-factor $\mathcal{G}_{f^{(2)}_{i_k,1}}$ as explained above. 
%  Thus from $G$ we can extract a $1$-factor and the claim is proved.
\vspace{0.5mm}

$\mathcal{G}_f = (\mathcal{G}_f\backslash\mathcal{G}_{f^{(2)}}) \circ \mathcal{G}_{f^{(2)}} = (\mathcal{G}_f\backslash\mathcal{G}_{f^{(2)}}) \circ \sum_{i=1}^p\sum_{k=1}^{m_i} a_ib_{i_k}\mathcal{G}_{f^{(2)}_{i_k}} = \sum_{i=1}^p\sum_{k=1}^{m_i}a_ib_{i_k}(\mathcal{G}_f\backslash\mathcal{G}_{f^{(2)}}) \circ \mathcal{G}_{f^{(2)}_{i_k}}) = \sum_{i=1}^p\sum_{k=1}^{m_i}a_ib_{i_k} \mathcal{G}_{f_{i_k}{''}}$, where each $\mathcal{G}_{f_{i_k}{''}}$ is a $ks$-regular graph and for each $\mathcal{G}_{f_{i_k}{''}}$ we get a $s$-factor by combining any $\frac{s-1}{2}$ number of $2$-factors of $(\mathcal{G}_f\backslash\mathcal{G}_{f^{(2)}}) \circ (\mathcal{G}_{f^{(2)}_{i_k}} \backslash \mathcal{G}_{f^{(2)}_{i_k,1}})$ (which is a $ks-1$-regular graph with even loops) with the $1$-factor  $\mathcal{G}_{f^{(2)}_{i_k,1}}$.

% So, $\mathcal{G}_f$ can be written as a linear combination of $ks$-regular graphs and for each of the summands in $\mathcal{G}_f$ we get a $s$-factor by combining any $\frac{s-1}{2}$ number of $2$-factors of $\mathcal{G}_f \backslash \mathcal{G}_{f^{(2)}}$ and one $1$-factor of $\mathcal{G}_{f^{(2)}_{i_k}}$.

\underline{\textbf{$k$ is odd and $n$ is even.}}

%\underline{\textbf{.}}

In this case $\mathcal{G}_f$ is $ks$-regular with even number of vertices and even number of loops. 
We form a new graph $\mathcal{G}_{\tilde{f}}$ by doubling the vertex set: for each vertex $v_i$ we associate two vertices $M_i$ and $N_i$, i.e., the vertex set of $\mathcal{G}_{\tilde{f}}$ is: 
$$\textrm{Vert}(\mathcal{G}_{\tilde{f}})=\{M_1,\ldots,M_n,N_1,\ldots,N_n\}.$$ 
 For each edge $(v_i,v_j)$ of $\mathcal{G}$, we associate two edges $(M_i,N_j)$ and $(M_j,N_i)$ in $\mathcal{G}_{\tilde{f}}$. For each loop $(v_i,v_i)$, we associate an edge $(M_i,N_i)$ in $\mathcal{G}_{\tilde{f}}$. Note that $\mathcal{G}_{\tilde{f}}$ is $ks$-regular and bi-partite between $M$ and $N$. So, by \cref{1factor}, it has a 1-factor, say $\tilde{\Delta}$, in $\mathcal{G}_{\tilde{f}}$.
 
From $\tilde{\Delta}$ we construct another graph $\Delta$ as follows:\\
(a) $\Delta$ has $n$ vertices, denoted by $\{1,2,\ldots,n\}$.\\
(b) for each edge $(M_i,N_j)$ in $\tilde{\Delta}$, we associate an edge $(i,j)$ in $\Delta$.\\
(c) for each edge $(M_i,N_i)$ in $\tilde{\Delta}$, we associate a loop $(i,i)$ in $\Delta$.\\
Note that the loops are disjoint components in $\Delta$ and the remaining graph ($\Delta\backslash$ \{loops\}) is $2$-regular consisting of cycles.
%So, $\Delta$ is almost a subgraph of $G$.
However, we may have both $(M_i,N_j)$ and $(M_j,N_i)$ are edges of $\tilde{\Delta}$. This may result in two occurrences of the edge $(i,j)$ in $\Delta$ but only one in $\mathcal{G}_f$. So, this type of component is a $2$-cycle. Note that $\Delta \backslash$ \{$2$-cycles\} is a subgraph of $\mathcal{G}_f$. If we pick $1$-factor from each of the $2$-cycles then the graph obtained by taking union of $ \Delta \backslash \{2\mbox{-cycles}\} $ and the chosen $1$-factors of $2$-cycles is a spanning subgraph of $\mathcal{G}_f$ and we denote it by $\mathcal{G}_{f'}$. Now we apply Pl\"{u}cker relations to write $\mathcal{G}_{f'}$ as a linear combination of graphs such that from each of the summands we can extract a $1$-factor in the following way and in each case we get a $s$-factor from each of the summands of $\mathcal{G}_f$.

(1) If $\mathcal{G}_{f'}$ has some loops then we consider two cases:

(a) If $\mathcal{G}_{f'}$ has even number of odd cycles then $\mathcal{G}_{f'}$ has even number of loops. Now we use Pl\"{u}cker relations repeatedly to merge two odd cycles into an even cycle and write $\mathcal{G}_{f'} =\sum_i a_i\mathcal{G}_{f'_{i}}$, $a_i \in \mathbb{Z}$, where $\mathcal{G}_{f'_{i}}$ is a disjoint union of even cycles, loops and $1$-factors of $2$-cycles of $\mathcal{G}_{f'}$. Now we can extract a $1$-factor $\mathcal{G}_{f'_{i,1}}$ from each $\mathcal{G}_{f'_{i}}$ as explained above.
%Now since each $\mathcal{H}_f^{i}$ has even number of loops, each $1$-factor $\mathcal{H}_{f,1}^{i}$ has also even number of loops.

%\textcolor{red}{Now we combine $\mathcal{G}_f\backslash\mathcal{H}_f$ with $\mathcal{H}_f^{i} \backslash \mathcal{H}_{f,1}^{i}$ and we get a $ks-1$-regular graph with even loops.}
%% Thus from each of the summand of $\mathcal{G}_f$ we get a $1$-factor which has even number of loops. Thus after extracting the $1$-factor from each summand of $\mathcal{G}_f$, the remaining summand is a $(ks-1)$-regular graph with even number of loops.
% \textcolor{red}{Since $ks-1$ is even so we use \cref{petersen} to get a $(s-1)$-factor. Thus combining this $(s-1)$-factor with $\mathcal{H}_{f,1}^{i}$ we get a $s$-factor of each summand.}
 
  $\mathcal{G}_f = (\mathcal{G}_f\backslash\mathcal{G}_{f'}) \circ \mathcal{G}_{f'} = (\mathcal{G}_f\backslash\mathcal{G}_{f'}) \circ \sum_i a_i\mathcal{G}_{f'_{i}} = \sum_i a_i((\mathcal{G}_f\backslash\mathcal{G}_{f'}) \circ \mathcal{G}_{f'_{i}}) = \sum_i a_i \mathcal{G}_{f_i{''}}$, where each $\mathcal{G}_{f_i{''}}$ is a $ks$-regular graph and for each $\mathcal{G}_{f_i{''}}$ we get a $s$-factor by combining any $\frac{s-1}{2}$ number of $2$-factors of $(\mathcal{G}_f\backslash\mathcal{G}_{f'}) \circ (\mathcal{G}_{f'_{i}} \backslash \mathcal{G}_{f'_{i,1}})$ (which is a $ks-1$-regular graph with even loops) with the $1$-factor $\mathcal{G}_{f'_{i,1}}$.

(b) If $\mathcal{G}_{f'}$ has odd number of odd cycles then $\mathcal{G}_{f'}$ also has odd number of loops. Now we use Pl\"{u}cker relations repeatedly to merge two odd cycles into an even cycle and write $\mathcal{G}_{f'}=\sum_i a_i\mathcal{G}_{f'_i}$, $a_i \in \mathbb{Z}$ where $\mathcal{G}_{f'_i}$ is a disjoint union of even cycles, loops, one odd cycle and $1$-factors of $2$-cycles of $\mathcal{G}_{f'}$. 

 Since $\mathcal{G}_{f'}$ has at least one loop, each $\mathcal{G}_{f'_i}$ also has at least one loop. So to get $\mathcal{G}_{f'_i}$ free of the odd cycle we apply Pl\"{u}cker relation on an edge of the odd cycle and one of the loops to write $\mathcal{G}_{f'_i}=\sum_{k} b_{i_k}\mathcal{G}_{f'_{i_k}}$, $b_{i_k} \in \mathbb{Z}$, where $\mathcal{G}_{f'_{i_k}}$ is a disjoint union of even cycles, loops (even in number) and one odd path starting with a loop and $1$-factors of $2$-cycles of $\mathcal{G}_{f'}$. We now extract a $1$-factor from the odd path by taking alternate edges and we extract $1$-factors from the other components of the linear combination as explained above. Thus we extract a $1$-factor $\mathcal{G}_{f'_{i_k,1}}$ from each $\mathcal{G}_{f'_{i_k}}$. 

%Now we combine $\mathcal{G}_f\backslash\mathcal{H}_f$ with $\mathcal{H}_f^{i_k} \backslash \mathcal{H}_{f,1}^{i_k}$ and we get a $ks-1$-regular graph with even loops.

 In this case $\mathcal{G}_f$ can be written as a linear combination of $ks$-regular graphs and for each of the summands in $\mathcal{G}_f$ we can get a $s$-factor as explained above.

%\textcolor{red}{So, $\mathcal{G}_f$ can be written as a linear combination of $ks$-regular graphs and for each of the summands in $\mathcal{G}_f$ we get a $s-1$-factor $\mathcal{G}_f^{i}$ by combining any $\frac{s-1}{2}$ number of $2$-factors of $(\mathcal{G}_f\backslash\mathcal{H}_f) \circ (\mathcal{H}_f^{i_k} \backslash \mathcal{H}_{f,1}^{i_k})$ (which is a $ks-1$-regular graph with even loops). Now $\mathcal{G}_f^{i} \circ \mathcal{H}_{f,1}^{i_k}$ is a $s$-factor of each of the summands of $\mathcal{G}_f$.}

%\textcolor{red}{ Since $ks-1$ is even we use \cref{petersen} to get a $(s-1)$-factor. Thus combining this $(s-1)$-factor with $\mathcal{H}_{f,1}^{i_k}$ we get a $s$-factor of each of the summand.}

 (2) If $\mathcal{G}_{f'}$ does not contain any loop then since the number of vertices is even, there are even number of odd cycles in $\mathcal{G}_{f'}$. Now we use Pl\"{u}cker relations repeatedly to merge two odd cycles into an even cycle and write $\mathcal{G}_{f'}=\sum_i a_i\mathcal{G}_{f'_i}$, $a_i \in \mathbb{Z}$ where $\mathcal{G}_{f'_i}$ is a disjoint union of even cycles and $1$-factors of $2$-cycles of $\mathcal{G}_{f'}$. We then extract a $1$-factor from each $\mathcal{G}_{f'_i}$ as explained above.
 
%  So we get a $1$-factor $\mathcal{H}_{f,1}^{i_k}$ of $\mathcal{H}_f^{i}$ which does not contain any loop.}
%  Now we combine $\mathcal{G}_f\backslash\mathcal{H}_f$ with $\mathcal{H}_f^{i} \backslash \mathcal{H}_{f,1}^{i}$ and we get a $ks-1$-regular graph with even number of loops. Now we use \cref{petersen} to get a $(s-1)$-factor. Thus combining this $(s-1)$-factor with $\mathcal{H}_{f,1}^{i}$ we get a $s$-factor.
 
 So, $\mathcal{G}_f$ can be written as a linear combination of $ks$-regular graphs and for each of the summands in $\mathcal{G}_f$ we can get a $s$-factor as explained above.
 
%  { we get a $s$-factor by combining any $\frac{s-1}{2}$ number of $2$-factors of the combination $\mathcal{G}_f\backslash\mathcal{H}_f$ with $\mathcal{H}_f^{i} \backslash \mathcal{H}_{f,1}^{i}$ (which is a $ks-1$-regular graph with even loops) and the $1$-factor $\mathcal{H}_{f,1}^{i}$.}

 \underline{\textbf{$k$ and $n$ both are odd.}}
 
 In this case $\mathcal{G}_f$ is a $ks$-regular graph with odd number of loops. As in the case where $k$ is odd and $n$ is even in this case also we get a bipartite graph, with bipartitions $M$ and $N$, which is $1$-factorable. Note that one of the factors contains odd number of edges of type $(M_i, N_i)$. So the associated graph $\Delta$ contains odd number of loops. Note that $\Delta \backslash \{\mbox{2-cycles}\}$ is a subgraph of $\mathcal{G}_f$. If we pick a $1$-factor from each of the $2$-cycles then the graph obtained by taking union of $\Delta\backslash\{\mbox{2-cycles}\}$ and the chosen $1$-factors of $2$-cycles is a spanning subgraph of $\mathcal{G}_f$ and we denote it by $\mathcal{G}_{f'}$. Note that $\mathcal{G}_{f'}$ is a disjoint union of even number of odd cycles, even cycles, odd number of loops and $1$-factors of $2$-cycles. Now we apply Pl\"{u}cker relations repeatedly to merge two odd cycles into an even cycle and write $\mathcal{G}_{f'} = \sum a_i\mathcal{G}_{f'_i}$, where $a_i \in \mathbb{Z}$ and each $\mathcal{G}_{f'_i}$ is a disjoint union of even cycles, loops and $1$-factors of $2$-cycles. Then we can extract a $1$-factor $\mathcal{G}_{f'_{i,1}}$ from each $\mathcal{G}_{f'_i}$ as explained above. 
% Now we combine $\mathcal{G}_f\backslash\mathcal{H}_f$ with $\mathcal{H}_f^{i} \backslash \mathcal{H}_{f,1}^{i}$ and we get a $ks-1$-regular graph with even number of loops. Now we use \cref{petersen} to get a $(s-1)$-factor. Thus combining this $(s-1)$-factor with $\mathcal{H}_{f,1}^{i}$ we get a $s$-factor.

 In this case also $\mathcal{G}_f$ can be written as a linear combination of $ks$-regular graphs and for each of the summands in $\mathcal{G}_f$ we can get a $s$-factor as explained above.
 
%\textcolor{red}{So, $\mathcal{G}_f$ can be written as a linear combination of $ks$-regular graphs and for each of the summands in $\mathcal{G}_f$ we get a $s$-factor by combining any $\frac{s-1}{2}$ number of $2$-factors of the combination $\mathcal{G}_f\backslash\mathcal{H}_f$ with $\mathcal{H}_f^{i} \backslash \mathcal{H}_{f,1}^{i}$ (which is a $ks-1$-regular graph) and the $1$-factor $\mathcal{H}_{f,1}^{i}$.} 
 
% \textcolor{red}{Since $ks-1$ is even we use \cref{petersen} to get a $(s-1)$-factor. Thus, combining this $(s-1)$-factor with the obtained $1$-factor we get a $s$-factor.}

Now using the $s$-factors that we have obtained in each of the above cases we will get a $s$-factor with $nr_1$ number of loops with which the associated monomial lies in $R_1$.  

 We interchange loops and edges between the $s$-factor and the $(k-1)s$-factor without interchanging the degree of the vertices so that the monomial associated to the new $s$-factor lies in $R_1$. For example, (1) if $(v_i,v_i)$ and $(v_j,v_j)$ are two loops in the $s$-factor and $(v_i,v_j)$ is an edge in the $(k-1)s$-factor then we can interchange them and (2) if $\{(v_i,v_j), (v_k,v_k), (v_l,v_l)\}$ is a set of an edge and two loops in the $s$-factor and $\{(v_i,v_k), (v_j,v_l)\}$ is a set of edges in the $(k-1)s$-factor then we can interchange them. We shall do this interchange for all possible loops and edges in the $s$-factor and the $(k-1)s$-factor.
 
  If interchange between loops and edges is not possible then we use Pl\"{u}cker relation on the factors repeatedly (possibly multiple times) to get a set of graphs where interchange between loops and edges is possible. The Pl\"{u}cker relation on the edge $(v_i,v_j)$ and loop $(v_k,v_k)$ (w.l.o.g we take  $\{i,j,k:i<j<k\}$), is $p_{ij}p_k=p_{ik}p_j - p_{jk}p_i$, and the Pl\"{u}cker relation on two edges $(v_i,v_j)$ and $(v_k,v_l)$ (w.l.o.g $\{i,j,k,l: i < j < k < l\}$), is $p_{ij}p_{kl}=p_{ik}p_{jl} - p_{il}p_{jk}$. We illustrate this possibility by an example given below.
 
  Now using induction on the number of loops we get a $s$-factor with which the associated monomial lies in $R_1$.
  
  So we conclude that $R$ is generated by $R_1$. Hence, the quotient $T\backslash\backslash(G/P)_T^{ss}(\mathcal L_{n\varpi})$ is projectively normal with respect to the descent of the line bundle $\mathcal{L}_{n\varpi}$.
    \end{proof}
    
 Here we give an example where interchange between loops and edges in the above theorem is not possible. Then we use Pl\"{u}cker relation on the factors to get a set of graphs where interchange between loops and edges is possible. Let us consider $\lambda= 6(\varpi_1+2\varpi_2)$ and $k=3$. So $s = 5$ and $nr_1= 6$. After using the above procedure, suppose we get a $5$-factor with which the associated monomial is\\
\centerline{$p_{12}^2p_{14}^3p_{24}^2p_{25}p_{35}^4p_{36}p_{6}^4$ (Figure $4$),}
and another graph which is a $10$-factor, with which the associated monomial is\\ \centerline{$p_{12}^2p_{13}^5p_{14}^3p_{24}^7p_{25}p_{35}^5p_{5}^4
p_{6}^{10}$ (Figure $1$).}
Here directly we can not interchange loops and edges between the factors. So we apply Pl\"{u}cker relation on the edge $(v_2,v_4)$ and the loop $(v_5,v_5)$ in Figure $1$, and obtain\\
\centerline{$p_{12}^2p_{13}^5p_{14}^3p_{24}^7p_{25}p_{35}^5p_{5}^4
p_{6}^{10}=p_{12}^2p_{13}^5p_{14}^3p_{24}^6p_{25}^2
p_{35}^5p_{5}^3p_4p_{6}^{10}-p_{12}^2p_{13}^5p_{14}^3p_{24}^6p_{25}p_{35}^5p_{45}
p_2p_{5}^3p_{6}^{10}$.}
The graph associated with the monomial $p_{12}^2p_{13}^5p_{14}^3p_{24}^6p_{25}^2
p_{35}^5p_{5}^3p_4p_{6}^{10}$ is in Figure 2, and the graph associated with the monomial $p_{12}^2p_{13}^5p_{14}^3p_{24}^6p_{25}p_{35}^5p_{45}
p_2p_{5}^3p_{6}^{10}$ is in Figure 3.

Now we can do the interchange as follows:
 
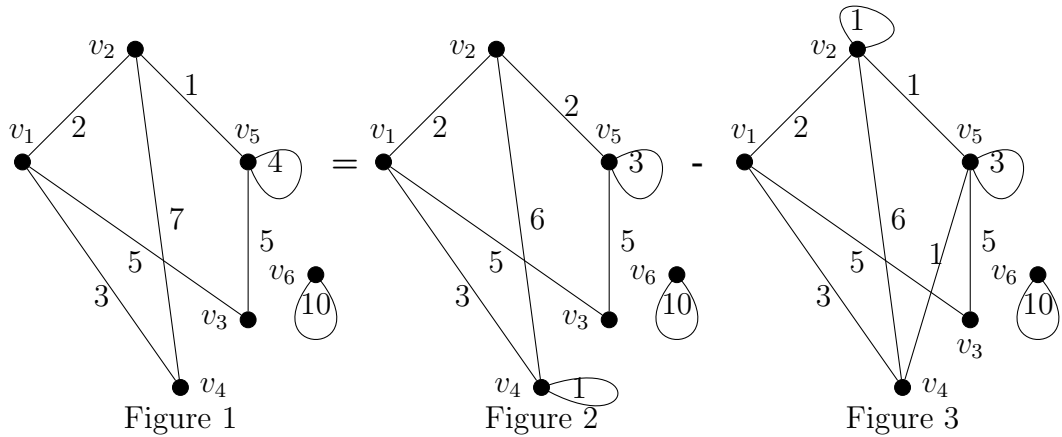
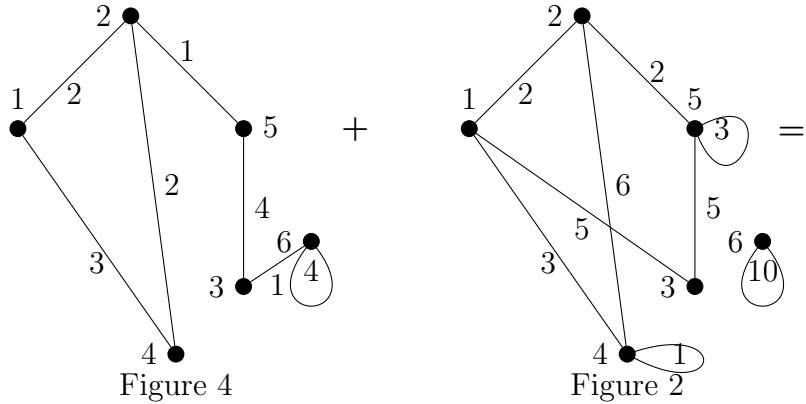
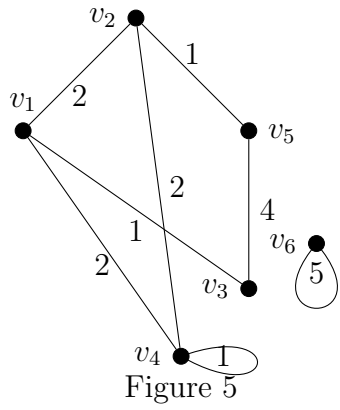
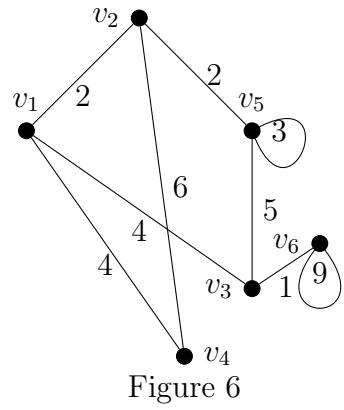
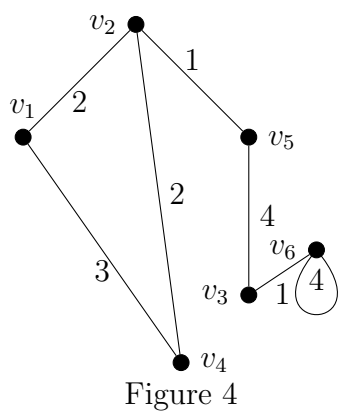
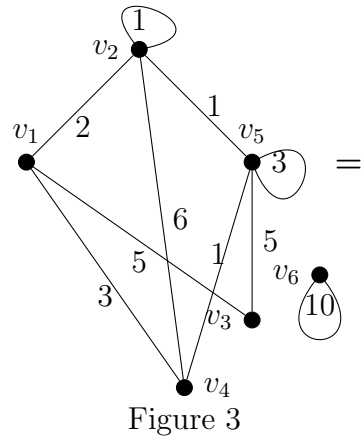
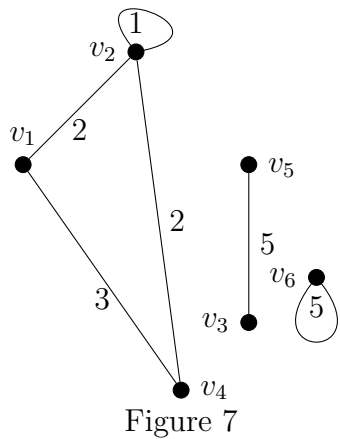
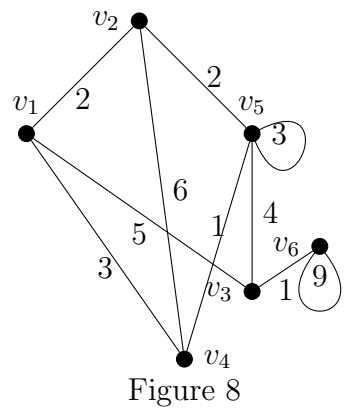
\[\begin{tikzpicture}[scale=3, every loop/.style={}]

    \vertex[fill] (1) at (0, 1)  [label=above: $v_1$] {}; 
 \vertex [fill](2) at (.5,1.5) [label=left: $v_2$] {};
\vertex[fill] (3) at (1,.3) [label=left: $v_3$]{};
  \vertex[fill] (4) at (.7,0) [label=right: $v_4$][label=below: \vspace{1cm}\mbox{Figure $1$}] {};   
    \vertex[fill] (5) at (1, 1) [label=above: $v_5$][label=right: 4 \hspace{.5cm}\mbox{$\textbf{=}$}]  {};  
    \vertex [fill](6) at (1.3,.5) [label=left: $v_6$][label=below: 10] {};
    \vertex[fill] (7) at (1.6, 1)  [label=above: $v_1$] {}; 
 \vertex [fill](8) at (2.1,1.5) [label=left: $v_2$] {};
\vertex[fill] (9) at (2.6,.3) [label=left: $v_3$]{};
  \vertex[fill] (10) at (2.3,0) [label=right:\vspace{.2cm} 1][label=left: $v_4$][label=below: \vspace{1cm}\mbox{Figure $2$}] {};   
    \vertex[fill] (11) at (2.6, 1) [label=above: $v_5$][label=right:3 \hspace{.5cm}\mbox{$\textbf{-}$}]  {};  
    \vertex [fill](12) at (2.9,.5) [label=left: $v_6$][label=below: 10] {}; 
 \vertex[fill] (13) at (3.2, 1)  [label=above: $v_1$] {}; 
 \vertex [fill](14) at (3.7,1.5) [label=left: $v_2$] [label=above: 1]{};
\vertex[fill] (15) at (4.2,.3) [label=below: $v_3$]{};
  \vertex[fill] (16) at (3.9,0) [label=right: $v_4$][label=below: \vspace{1cm}\mbox{Figure $3$}] {};   
    \vertex[fill] (17) at (4.2, 1) [label=above: $v_5$][label=right:3]  {};  
    \vertex [fill](18) at (4.5,.5) [label=left: $v_6$][label=below: 10] {};     
 \draw (1.3,.5)  to[in=-50,out=-130,loop] (5,1);
 \draw (2.9,.5)  to[in=-50,out=-130,loop] (5,1);
   \draw (2.3,0)  to[in=20,out=-30,loop] (5,1);
  \draw (2.6,1)  to[in=30,out=-70,loop] (2.6,1);
  \draw (1,1)  to[in=30,out=-70,loop] (1,1);
  \draw (3.7,1.5)  to[in=130,out=10,loop] (4,1);
  \draw (4.2,1)  to[in=30,out=-70,loop] (4.2,1);
  \draw (4.5,.5)  to[in=-50,out=-130,loop] (4.5,.5);
    \path

        (1) edge node [below] {5} (3)
        (1) edge node [below] {2} (2)
        (1) edge node [below] {3} (4)
        (2) edge node [right] {7} (4)
        (2) edge node [above] {1} (5)
        (5) edge node [right] {5} (3)
        %(3) edge node [right] {1} (6)
        (7) edge node [below] {2} (8)
        (7) edge node [below] {5} (9)
        (7) edge node [below] {3} (10)
        (8) edge node [right] {6} (10)
        (8) edge node [right] {2} (11)
        (9) edge node [right] {5} (11)
        (13) edge node [below] {5} (15)
        (13) edge node [below] {2} (14)
        (13) edge node [below] {3} (16)
        (14) edge node [right] {6} (16)
        (14) edge node [above] {1} (17)
        (17) edge node [right] {5} (15)
        (16) edge node [above] {1} (17)
        ;  
        \end{tikzpicture}\]
      
\[\begin{tikzpicture}[scale=3, every loop/.style={}]

    \vertex[fill] (1) at (0, 1)  [label=above: 1] {}; 
 \vertex [fill](2) at (.5,1.5) [label=left: 2] {};
\vertex[fill] (3) at (1,.3) [label=left: 3]{};
  \vertex[fill] (4) at (.7,0) [label=left: 4][label=below: \vspace{1cm}\mbox{Figure $4$}] {};   
    \vertex[fill] (5) at (1, 1) [label=right: 5 \hspace{.7cm}\mbox{$\textbf{+}$}]  {};  
    \vertex [fill](6) at (1.3,.5) [label=left: 6][label=below: 4] {};
   \vertex[fill] (7) at (2, 1)  [label=above: 1] {}; 
 \vertex [fill](8) at (2.5,1.5) [label=left: 2] {};
\vertex[fill] (9) at (3,.3) [label=left: 3]{};
  \vertex[fill] (10) at (2.7,0) [label=left: 4][label=right:\hspace{.2cm} 1][label=below: \vspace{1cm}\mbox{Figure $2$}] {};   
    \vertex[fill] (11) at (3, 1) [label=above: 5][label=right:3 \hspace{.5cm}\mbox{$\textbf{=}$}]  {};  
    \vertex [fill](12) at (3.3,.5) [label=left: 6][label=below: 10] {}; 
    \draw (1.3,.5)  to[in=-50,out=-130,loop] (1.5,1); 
 \draw (3.3,.5)  to[in=-50,out=-130,loop] (5,1);
 \draw (2.7,0)  to[in=20,out=-30,loop] (5,1);
 \draw (3,1)  to[in=30,out=-70,loop] (3,1);
    \path

        (1) edge node [below] {3} (4)
        (1) edge node [below] {2} (2)
        (2) edge node [right] {2} (4)
        (2) edge node [above] {1} (5)
        (5) edge node [right] {4} (3)
        (3) edge node [below] {1} (6)
        (7) edge node [below] {2} (8)
        (7) edge node [below] {5} (9)
        (7) edge node [below] {3} (10)
        (8) edge node [right] {6} (10)
        (8) edge node [right] {2} (11)
        (9) edge node [right] {5} (11)
        ;  
\end{tikzpicture}\]

\[\begin{tikzpicture}[scale=3, every loop/.style={}]

    \vertex[fill] (1) at (0, 1)  [label=above: $v_1$] {}; 
 \vertex [fill](2) at (.5,1.5) [label=left: $v_2$] {};
\vertex[fill] (3) at (1,.3) [label=left: $v_3$]{};
  \vertex[fill] (4) at (.7,0) [label=right: \hspace{.2cm}1][label=left: $v_4$][label=below:  \mbox{Figure $5$}] {};   
    \vertex[fill] (5) at (1, 1) [label=right: $v_5$ \hspace{.7cm}\mbox{$\textbf{+}$}]  {};  
    \vertex [fill](6) at (1.3,.5) [label=left: $v_6$][label=below: 5] {};
   \vertex[fill] (7) at (2, 1)  [label=above: $v_1$] {}; 
 \vertex [fill](8) at (2.5,1.5) [label=left: $v_2$] {};
\vertex[fill] (9) at (3,.3) [label=left: $v_3$]{};
  \vertex[fill] (10) at (2.7,0) [label=right: $v_4$][label=below: \vspace{1cm}\mbox{Figure $6$}] {};   
    \vertex[fill] (11) at (3, 1) [label=above: $v_5$][label=right:3]  {};  
    \vertex [fill](12) at (3.3,.5) [label=left: $v_6$][label=below: 9] {}; 
    \draw (1.3,.5)  to[in=-50,out=-130,loop] (1.5,1); 
    \draw (.7,0)  to[in=20,out=-30,loop] (.7,0);
 \draw (3.3,.5)  to[in=-50,out=-130,loop] (5,1);
 %\draw (2.7,0)  to[in=-50,out=-130,loop] (5,1);
 \draw (3,1)  to[in=30,out=-70,loop] (3,1);
    \path

        (1) edge node [below] {2} (4)
        (1) edge node [below] {2} (2)
        (2) edge node [right] {2} (4)
        (2) edge node [above] {1} (5)
        (5) edge node [right] {4} (3)
        (1) edge node [below] {1} (3)
        (7) edge node [below] {2} (8)
        (7) edge node [below] {4} (9)
        (7) edge node [below] {4} (10)
        (8) edge node [right] {6} (10)
        (8) edge node [right] {2} (11)
        (9) edge node [right] {5} (11)
        (9) edge node [below] {1} (12)
        ;  
\end{tikzpicture}\]
\[\begin{tikzpicture}[scale=3, every loop/.style={}]

    \vertex[fill] (1) at (0, 1)  [label=above: $v_1$] {}; 
 \vertex [fill](2) at (.5,1.5) [label=left: $v_2$] {};
\vertex[fill] (3) at (1,.3) [label=left: $v_3$]{};
  \vertex[fill] (4) at (.7,0) [label=right: $v_4$][label=below: \vspace{1cm}\mbox{Figure $4$}] {};   
    \vertex[fill] (5) at (1, 1) [label=right: $v_5$ \hspace{.7cm}\mbox{$\textbf{+}$}]  {};  
    \vertex [fill](6) at (1.3,.5) [label=left: $v_6$][label=below: 4] {};
   \vertex[fill] (7) at (2, 1)  [label=above: $v_1$] {}; 
 \vertex [fill](8) at (2.5,1.5) [label=left: $v_2$] [label=above: 1]{};
\vertex[fill] (9) at (3,.3) [label=left: $v_3$]{};
  \vertex[fill] (10) at (2.7,0) [label=right: $v_4$][label=below: \vspace{1cm}\mbox{Figure $3$}] {};   
    \vertex[fill] (11) at (3, 1) [label=above: $v_5$][label=right:3 \hspace{.5cm}\mbox{$\textbf{=}$}]  {};  
    \vertex [fill](12) at (3.3,.5) [label=left: $v_6$][label=below: 10] {}; 
    \draw (1.3,.5)  to[in=-50,out=-130,loop] (1.5,1); 
 \draw (3.3,.5)  to[in=-50,out=-130,loop] (5,1);
 %\draw (2.7,0)  to[in=-50,out=-130,loop] (5,1);
 \draw (3,1)  to[in=30,out=-70,loop] (3,1);
 \draw (2.5,1.5)  to[in=130,out=10,loop] (4,1);
    \path

        (1) edge node [below] {3} (4)
        (1) edge node [below] {2} (2)
        (2) edge node [right] {2} (4)
        (2) edge node [above] {1} (5)
        (5) edge node [right] {4} (3)
        (3) edge node [below] {1} (6)
        (7) edge node [below] {2} (8)
        (7) edge node [below] {5} (9)
        (7) edge node [below] {3} (10)
        (8) edge node [right] {6} (10)
        (8) edge node [right] {1} (11)
        (9) edge node [right] {5} (11)
        (10) edge node [above] {1} (11)
        ;  
\end{tikzpicture}\]
\vspace{-1mm}

\vspace{-1mm}
\[\begin{tikzpicture}[scale=3, every loop/.style={}]

    \vertex[fill] (1) at (0, 1)  [label=above: $v_1$] {}; 
 \vertex [fill](2) at (.5,1.5) [label=left: $v_2$][label=above: 1] {};
\vertex[fill] (3) at (1,.3) [label=left: $v_3$]{};
  \vertex[fill] (4) at (.7,0) [label=right: $v_4$][label=below: \vspace{1cm}\mbox{Figure $7$}] {};   
    \vertex[fill] (5) at (1, 1) [label=right: $v_5$ \hspace{.7cm}\mbox{$\textbf{+}$}]  {};  
    \vertex [fill](6) at (1.3,.5) [label=left: $v_6$][label=below: 5] {};
   \vertex[fill] (7) at (2, 1)  [label=above: $v_1$] {}; 
 \vertex [fill](8) at (2.5,1.5) [label=left: $v_2$] {};
\vertex[fill] (9) at (3,.3) [label=left: $v_3$]{};
  \vertex[fill] (10) at (2.7,0) [label=right: $v_4$][label=below: \vspace{1cm}\mbox{Figure $8$}] {};   
    \vertex[fill] (11) at (3, 1) [label=above: $v_5$][label=right:3]  {};  
    \vertex [fill](12) at (3.3,.5) [label=left: $v_6$][label=below: 9] {}; 
    \draw (1.3,.5)  to[in=-50,out=-130,loop] (1.5,1); 
   % \draw (.7,0)  to[in=-50,out=-130,loop] (.7,0);
 \draw (3.3,.5)  to[in=-50,out=-130,loop] (5,1);
 %\draw (2.7,0)  to[in=-50,out=-130,loop] (5,1);
 \draw (3,1)  to[in=30,out=-70,loop] (3,1);
 \draw (.5,1.5)  to[in=130,out=10,loop] (4,1);
    \path

        (1) edge node [below] {3} (4)
        (1) edge node [below] {2} (2)
        (2) edge node [right] {2} (4)
        %(2) edge node [above] {2} (5)
        (5) edge node [right] {5} (3)
       % (1) edge node [below] {1} (3)
       % (6) edge node [below] {1} (3)
        (7) edge node [below] {2} (8)
        (7) edge node [below] {5} (9)
        (7) edge node [below] {3} (10)
        (8) edge node [right] {6} (10)
        (8) edge node [right] {2} (11)
        (9) edge node [right] {4} (11)
        (9) edge node [below] {1} (12)
        (10) edge node [above] {1} (11)
        ;  
\end{tikzpicture}\]

(a) Interchange the pair of edges $\{(v_1,v_4), (v_3,v_6)\}$ in Figure 4  with the set $\{(v_1,v_3), (v_4,v_4),\\ (v_6,v_6)\}$ in Figure $2$, and obtain the graph in Figure 5 with which the associated monomial $p_{12}^2p_{13}p_{14}^2p_{24}^2p_{25}p_{35}^4p_4p_{6}^5$ lies in $R_1$ and the graph in Figure 6 with which the associated monomial lies in $R_2$.

(b) Interchange the pair of edges $\{(v_2,v_5), (v_3,v_6)\}$ in Figure 4 with the set $\{(v_3,v_5), (v_2,v_2),\\ (v_6,v_6)\}$ in Figure 3 to obtain the graph in Figure 7 with which the associated monomial $p_{12}^2p_{14}^3p_{24}^2p_{35}^5p_{2}p_{6}^5$ lies in $R_1$ and the graph in Figure 8 with which the associated monomial lies in $R_2$.
 
\begin{corollary}
 The GIT quotient of a Schubert variety and a Richardson variety in $SL_n/(P_{\alpha_1}\cap P_{\alpha_2})$ by a maximal torus $T$ is projectively normal with respect to the descent of the line bundle $\mathcal L_{n(r_1\varpi_1+r_2\varpi_2)}$.
 \end{corollary}
 
 \begin{proof}
 The proof is same as the proof of \cref{cor1}.
 \end{proof}

\section{$Spin_5$ and $Spin_7$}\label{secb3}

In this section for $G = Spin_5$ and $Spin_7$ and for a maximal parabolic subgroup $P$ we study projective normality of the quotient $T\backslash\backslash(G/P)$ with respect to the descent of a suitable line bundle on $G/P$.
 
 \textbf{Notation:}
We denote a Young tableau $\Gamma$ with rows $\textrm{Row}_1, \textrm{Row}_2, \ldots, \textrm{Row}_n$ by $\Gamma=(\textrm{Row}_1, \textrm{Row}_2, \ldots, \textrm{Row}_n)$.

\subsection{$Spin_5$}
Let $G=Spin_5$. Let $\varpi_1$ and $\varpi_2$ be the fundamental weights associated to the simple roots $\alpha_1$ and $\alpha_2$ respectively. Since $2\varpi_1 \in 2Q$ and $2\varpi_2 \in \mathbb{Z}\alpha_1+\mathbb{Z}2\alpha_2$, by \cref{shrawan}, the line bundles $\mathcal{L}_{2\varpi_1}$ and $\mathcal{L}_{2\varpi_2}$ descend to the quotients $T\backslash\backslash(G/P_{\alpha_1})^{ss}_T(\mathcal
{L}_{2\varpi_1})$ and $T\backslash\backslash(G/P_{\alpha_2})^{ss}_T(\mathcal
{L}_{2\varpi_2})$ respectively. We have,\\
\centerline{ $T\backslash\backslash(G/P_{\alpha_1})^{ss}_T(\mathcal{L}_{2\varpi_1}) \cong Proj(\oplus H^0(G/P_{\alpha_1}, \mathcal{L}_{2\varpi_1}^{\otimes k})^T)\cong Proj(\oplus_{k \in \mathbb{Z}_{\geq 0}}R_k)$,}\\
 where $R_k:=H^0(G/P_{\alpha_1}, \mathcal{L}_{2\varpi_1}^{\otimes k})^T$. By Theorem 3.2.4. the standard monomials $p_{\Gamma}$ form a basis of $R_k$, where $\Gamma$ is a standard Young tableau associated to the weight $2k\varpi_1$. The standard monomials in $R_k$ are of the form  $p_{\Gamma}$, where\\
 \centerline{ $\Gamma=(\underbrace{(1), \ldots, (1)}_{2q},\underbrace{(2), \ldots, (2)}_{2k-2q},\underbrace{(3), \ldots, (3)}_{2k-2q},\underbrace{(4), \ldots, (4)}_{2q})$,}\\ $0 \leq q \leq k$. So, the homogeneous coordinate ring of the quotient $T\backslash\backslash(G/P_{\alpha_1})_T^{ss}(\mathcal L_{2 \varpi_1})$ is generated by $p_{\Gamma_1}$ and $p_{\Gamma_2}$, where $\Gamma_1=((1),(1),(4),(4))$ and $\Gamma_2=((2),(2),(3),(3))$ as an algebra. Since $T\backslash\backslash(G/P_{\alpha_1})^{ss}_T(\mathcal
{L}_{2\varpi_1})$ is normal, it is projectively normal. In fact, in this case $T\backslash\backslash(G/P_{\alpha_1})^{ss}_T(\mathcal
{L}_{2\varpi_1}) \cong \mathbb{P}^1$. \\

For the quotient $T\backslash\backslash(G/P_{\alpha_2})^{ss}_T(\mathcal{L}_{2\varpi_2})$, the standard monomials in $R_k$ are of the form  $p_{\Gamma}$, $\Gamma=(\underbrace{(1,2), \ldots, (1,2)}_{q},\underbrace{(1,3), \ldots, (1,3)}_{k-q},\underbrace{(2,4), \ldots, (2,4)}_{k-q},\underbrace{(3,4), \ldots, (3,4)}_{q})$, $0 \leq q \leq k$. So, the homogeneous coordinate ring of the quotient $T\backslash\backslash(G/P_{\alpha_1})$ is generated by $p_{\Gamma_1}$ and $p_{\Gamma_2}$, where $\Gamma_1=((1,2),(3,4))$ and $\Gamma_2=((1,3),(2,4))$ as an algebra. Since the quotient $T\backslash\backslash(G/P_{\alpha_2})$ is normal,  it is projectively normal. In this case also $T\backslash\backslash(G/P_{\alpha_2})^{ss}_T(\mathcal
{L}_{2\varpi_2}) \cong \mathbb{P}^1$.\\

\subsection{$Spin_7$}
Let $G=Spin_7$ and $P_{\alpha_i}$ be the maximal parabolic subgroup subgroup associated to $\alpha_i$, $1 \leq i \leq 3$. In this case the line bundle $\mathcal{L}_{2\varpi_i}$ descends to the quotient $T\backslash\backslash(G/P_{\alpha_i})^{ss}_T(\mathcal
{L}_{2\varpi_i})$ for $1 \leq i \leq 2$ whereas $\mathcal{L}_{4\varpi_3}$ descends to the quotient $T\backslash\backslash(G/P_{\alpha_3})^{ss}_T(\mathcal
{L}_{4\varpi_3})$. 

We show that $T\backslash\backslash(G/P_{\alpha_1})^{ss}_T(\mathcal{L}_{2\varpi_1})$ and $T\backslash\backslash(G/P_{\alpha_3})^{ss}_T(\mathcal{L}_{4\varpi_3})$ are projectively normal with respect to the descent of the line bundles $\mathcal{L}_{2\varpi_1}$ and $\mathcal{L}_{4\varpi_3}$ respectively whereas we give a degree bound of the generators of the homogeneous coordinate ring of $T\backslash\backslash(G/P_{\alpha_2})^{ss}_T(\mathcal
{L}_{2\varpi_2})$.

We have $T\backslash\backslash(G/P_{\alpha_1})^{ss}_T(\mathcal{L}_{2\varpi_1}) \cong Proj(\oplus_{k \in \mathbb{Z}_{\geq 0}}R_k)$,
 where $R_k:=H^0(G/P_{\alpha_1}, \mathcal{L}_{2\varpi_1}^{\otimes k})^T$. The standard monomials $p_{\Gamma}$ form a basis of $R_k$, where $\Gamma$ is a standard Young tableau associated to the weight $2k\varpi_1$. The standard monomials in $R_k$ are of the form  $p_{\Gamma}$, where \\ \centerline{$\Gamma=(\underbrace{(1), \ldots, (1)}_{k_1},\underbrace{(2), \ldots, (2)}_{k_2},\underbrace{(3), \ldots, (3)}_{2k-k_1-k_2},\underbrace{(4), \ldots, (4)}_{2k-k_1-k_2},\underbrace{(5), \ldots, (5)}_{k_2},\underbrace{(6), \ldots, (6)}_{k_1})$,}\\ where $0 \leq k_1+k_2 \leq 2k$. So the homogeneous coordinate ring of the GIT quotient $T\backslash\backslash(G/P_{\alpha_1})^{ss}_T(\mathcal{L}_{2\varpi_1}) $
is generated by $p_{\Gamma_1}, p_{\Gamma_2}$ and $p_{\Gamma_3}$ as an algebra, where\\
 \centerline{$\Gamma_1=((1),(1),(6),(6))$, $\Gamma_2=((2),(2),(5),(5))$ and $\Gamma_3=((3),(3),(4),(4))$.}\\
  Since the quotient $T\backslash\backslash(G/P_{\alpha_1})^{ss}_T(\mathcal{L}_{2\varpi_1})$ is normal so it is projectively normal. In fact, in this case $T\backslash\backslash(G/P_{\alpha_1})^{ss}_T(\mathcal
{L}_{2\varpi_1}) \cong \mathbb{P}^2$. 

In the following theorem we give a degree bound of the generators of the homogeneous coordinate ring of the quotient $T\backslash\backslash(G/P_{\alpha_2})^{ss}_T(\mathcal
{L}_{2\varpi_2})$.  

\begin{theorem}\label{Theo4}
	The homogeneous co-ordinate ring of the quotient $T\backslash\backslash(G/P_{\alpha_2})^{ss}_T(\mathcal
{L}_{2\varpi_2})$ is generated by elements of degree at most $3$.
              \end{theorem}
              
              \begin{proof} 
              
              We have \[ T\backslash\backslash(G/P_{\alpha_2})_T^{ss}(\mathcal L_{2\varpi_2}) = Proj({\oplus_{k \in \mathbb{Z}_{\geq 0}}}H^0(G/P_{\alpha_2},\mathcal{L}^{\otimes k}_{2\varpi_2})^T)=Proj(\oplus_{k\in \mathbb{Z}_{\geq0}}R_k),\]where $R_k :=H^0(G/P_{\alpha_2},\mathcal{L}^{\otimes k}_{2\varpi_2})^T$. Let $f \in R_k$ be a standard monomial.
              
              We claim that $f = f_1.f_2$ where $f_1$ is in $R_1$ or $R_2$ or $R_3$.
               
                From the discussion in \cref{young}, the Young diagram associated to $f$ has the shape $p = (p_1, p_2) = (4k, 4k)$. So the Young tableau $\Gamma$ associated to this Young diagram has $4k$ rows and $2$ columns with strictly increasing rows and non-decreaing columns. Since $f$ is $T$-invariant, by \cref{zeroweight} we have, \begin{equation}\label{weight1} c_\Gamma(t)=c_\Gamma(7-t) \text{ for all } 1 \leq t \leq 6, 
              \end{equation}
           where $c_{\Gamma}(t) = \#\{t| t \in \Gamma\} $. Also from the discussion in \cref{young}, we have $(\textrm{Row}_{2i-1}, \textrm{Row}_{2i})$ is an admissible pair for all $1 \leq i \leq 2k$, where $\textrm{Row}_{i}$ denotes the $i$-th row of the tableau for all $1 \leq i \leq 4k$. We also have 
          $\mbox{if } t \in \textrm{Row}_{i} \mbox{ then } 7-t \notin \textrm{Row}_{i}, \mbox{ for all } 1 \leq t \leq 6 \mbox{ and for any } i, 1 \leq i \leq 4k$.
         
             Let $\textrm{Col}_{j}$ denotes the $j$-th column where $1 \leq j \leq 2$. Let $E_{i,j}$  be the $(i,j)$-th entry of the tableau $\Gamma$ and $N_{t,j} = \#\{ i| E_{i,j}=t\}$. 
            
       Since $(\textrm{Row}_{2i-1},\textrm{Row}_{2i})$ is admissible, either $\textrm{Row}_{2i-1}=\textrm{Row}_{2i}$ or $(\textrm{Row}_{2i-1},\textrm{Row}_{2i}) \in \{((1,3),(1,4)),((1,5),(2,6)), ((2,3),(2,4)), ((2,4),(3,5))\}$.

We consider $\textrm{Row}_1$. If $E_{1,1}=3$ then $ E_{1,2} \neq 4$. So $E_{1,2} = 5$ or $6$, a contradiction to \cref{weight1}. By a similar reason, $E_{1,1}$ can not be $4, 5$ or $6$. So $\textrm{Row}_1 \in \{(1,2), (1,3), (1,4), \\(1,5), (2,3), (2,4)\}$. 

(a)  Let $\textrm{Row}_1 = (2,4)$. Since $(\textrm{Row}_1, \textrm{Row}_2)$ is admissible so we have $\textrm{Row}_{2}=(2,4)$ or $(3,5)$. \\
            If $\textrm{Row}_2 = (3,5)$ then $5$ or $6$ has to appear in one of the rows below, which is a contradiction to \eqref{weight1}. \\
            If $\textrm{Row}_{2}=(2,4)$ then $(\textrm{Row}_{4k-1},\textrm{Row}_{4k})\in\{((3,5),(3,5)), ((3,5),(4,5)),((4,5),(4,5))\}$.
           
               If $(\textrm{Row}_{4k-1},\textrm{Row}_{4k})  = ((4,5),(4,5))$ then $c_{\Gamma}(4)+c_{\Gamma}(5)\geq 4k+2$ and hence $c_{\Gamma}(2)+c_{\Gamma}(3)\leq 4k-2$, a contradiction. By a similar reason $(\textrm{Row}_{4k-1},\textrm{Row}_{4k})$ can not be $((3,5),(4,5))$. If $(\textrm{Row}_{4k-1},\textrm{Row}_{4k}) = ((3,5), (3,5))$ then $p_{\Omega} \in  R_1$, where $\Omega=((2,4),(2,4),\\(3,5),(3,5))$ and is a factor of $f$.
              
            (b)  If $\textrm{Row}_{1}= (2,3)$ then $\textrm{Row}_{2}$ is either $(2,3)$ or $(2,4)$.\\
               If $\textrm{Row}_2=(2,3)$ then by a similar argument as above, $(\textrm{Row}_{4k-1},\textrm{Row}_{4k})$ has to be $((3,5),\\(4,5))$. Then $p_{\Omega} \in R_1$, where $\Omega=((2,3),(2,3),(4,5),(4,5))$ and is a factor of $f$.
               
               Similarly if $\textrm{Row}_2=(2,4)$  
then we have $p_{\Omega} \in R_1$, where $\Omega= ((2,3),(2,4),(3,5),(4,5))$ and is a factor of $f$. 
              
        (c) If $\textrm{Row}_1=(1,5)$ then $\textrm{Row}_2=(1,5)$. Then $(\textrm{Row}_{4k-1}, \textrm{Row}_{4k})\in \{((2,6),(2,6)),((3,6),\\(3,6)),((3,6), (4,6)),((4,6),(4,6)),((5,6),(5,6))\}$. By a similar argument as above $(\textrm{Row}_{4k-1},\\ \textrm{Row}_{4k})$ can not be any other pair except $((2,6),(2,6))$. Then $p_{\Omega}\in R_1$, where $\Omega=((1,5),(1,5),(2,6),(2,6))$ and is a factor of $f$.
        
        (d) If $\textrm{Row}_1 = (1,4)$ then  $\textrm{Row}_2=(1,4)$. By a similar argument as above $(\textrm{Row}_{4k-1}, \textrm{Row}_{4k})$ has to be $((3,6),(3,6))$. Then $p_{\Omega}\in R_1$ and is a factor of $f$, where $\Omega=((1,4),(1,4),(3,6),\\(3,6))$. 
        
  (e) If $\textrm{Row}_1=(1,2)$ then  $\textrm{Row}_2=(1,2)$ and $(\textrm{Row}_{4k-1}, \textrm{Row}_{4k}) \in \{((2,6),(2,6)),((3,6),\\(3,6)),((3,6), (4,6)), ((4,6),(4,6)),((5,6),(5,6))\}$. By a similar argument as above \\$(\textrm{Row}_{4k-1}, \textrm{Row}_{4k})$ is either $((5,6),(5,6))$ or $((4,6),(4,6))$.\\
     If $(\textrm{Row}_{4k-1}, \textrm{Row}_{4k}) = ((5,6),(5,6))$ then $p_{\Omega} \in R_1$ and is a factor of $f$, where $\Omega=((1,2),(1,2),(5,6),(5,6))$.\\
      Now assume $(\textrm{Row}_{4k-1}, \textrm{Row}_{4k}) = ((4,6),(4,6))$. Let $N_{1,1}=N_{6,2} = m_1$. By the admissibility property $m_1$ is even.
      
      \textbf{Case 1:} Let $m_1=2$. Note that $\textrm{Row}_i$ has either $2$ or $5$ as an entry for all $3 \leq i \leq 4k-2$. Since $E_{1,2}=E_{2,2} = 2$ we have $c_{\Gamma}(2)=c_{\Gamma}(5)= 2k-1$. Similarly, we have $c_{\Gamma}(3)=c_{\Gamma}(4)= 2k-1$. Since $E_{4k,1}=4$ we have $E_{i,2}=5$, for all $2k \leq i \leq 4k-2$. Also since $c_{\Gamma}(4)=2k-1$ we have $\textrm{Row}_{2k} = \textrm{Row}_{2k+1} = (3,5)$. Since the pairs $(\textrm{Row}_{2k-1}, \textrm{Row}_{2k})$ and $(\textrm{Row}_{2k+1}, \textrm{Row}_{2k+2})$ are admissible we have $\textrm{Row}_{2k-1}=(2,4)$ and $\textrm{Row}_{2k+2}=(3,5)$. Then  $p_{\Omega} \in R_2$ and is a factor of $f$, where $\Omega=((1,2),(1,2),(2,4),(3,5),(3,5),(3,5),(4,6),(4,6))$.
      
      \textbf{Case 2:} Let $m_1=2k$. We have $E_{i,1}=1$ for all $1 \leq i \leq 2k$ and $E_{i,2}=6$ for all $2k+1 \leq i \leq 4k$. Note that $E_{2k-1,2}=E_{2k,2}=5$. Then $E_{i,1}=E_{i+1,1}=3$, for some $i$, $2k+1 \leq i \leq 4k$. Then $p_{\Omega} \in R_2$ and is a factor of $f$, where $\Omega=((1,2),(1,2),(1,5),(1,5),(3,6),(3,6),(4,6),(4,6)) $. 
      
      \textbf{Case 3:} 	Let $4 \leq m_1 \leq 2k-2$. Note that  $k \geq 3$. 
 
Since $E_{1,2}=E_{2,2}=2$ and $\textrm{Row}_i$ contains either $2$ or $5$ as an entry for $m_1+1 \leq i \leq 4k-m_1$ we have $N_{5,2} \geq 2k-m_1+1$. Hence, $\textrm{Row}_{2k}=\textrm{Row}_{2k+1} = (3,5)$. Let $N_{2,1}=l$. 
      
     (i) If $l=0$ then  $E_{i,2}=5$ for all $m_1+1 \leq i \leq 4k-m_1$ and hence $N_{2,2} \geq 4$. So we have $\textrm{Row}_{3}=\textrm{Row}_{4}=(1,2)$. Since $c_{\Gamma}(4) \leq 2k-1$ we have $\textrm{Row}_{2k-1}=(3,5)$.\\
     Since $(\textrm{Row}_{2k+1},\textrm{Row}_{2k+2})$ is admissible we have $\textrm{Row}_{2k+2}$ is either $(3,5)$ or $(4,5)$. If $\textrm{Row}_{2k+2}=(4,5)$ then $c_{\Gamma}(4)=2k-1$ whereas  $c_{\Gamma}(3) \leq (2k-4)+1=2k-3$, a contradiction. Hence $\textrm{Row}_{2k+2}=(3,5)$.\\
     We claim that $N_{4,1} \geq 4$. If not then $N_{4,1} \leq 3$. In this case $c_{\Gamma}(3) \geq 2k-3+2=2k-1$ whereas $c_{\Gamma}(4) \leq 2k-6+3=2k-3$, a contradiction. Hence, $N_{4,1} \geq 4$.  So we have $\textrm{Row}_{4k-3}=\textrm{Row}_{4k-3}=(4,6)$.
    Then $p_{\Omega}\in R_3$ and is a factor of $f$, where\\
     \centerline{ $\Omega=((1,2),(1,2),(1,2),(1,2),(3,5),(3,5),(3,5),(3,5),(4,6),(4,6),(4,6),(4,6))$.}  
     
     (ii) Let $l=1$. Then $\textrm{Row}_{m_1+1}=(2,4)$. Since $(\textrm{Row}_{m_1+1},\textrm{Row}_{m_1+2})$ is admissible and $c_{\Gamma}(4) \leq 2k-1$ we have
      $\textrm{Row}_{m_1+2}=(3,5)$. Again since $c_{\Gamma}(4) \leq 2k-1$ we have $\textrm{Row}_{2k+1} = \textrm{Row}_{2k+2}= (3,5)$. Then $p_{\Omega} \in R_2$ and is a factor of $f$, where\\ \centerline{$\Omega=((1,2),(1,2),(2,4),(3,5),(3,5),(3,5),(4,6),(4,6))$.}
      
    (iii)  Let $l \geq 2$.
    
    Then the rows of $\Gamma$ containing $2$ as the first entry are either $(2,3)$ or $(2,4)$.\\
    If $\Gamma$ has at least two rows equal to $(2,4)$ then $p_{\Omega} \in R_1$ and is a factor of $f$, where $\Omega=((2,4),(2,4),(3,5),(3,5))$. \\
    If $\Gamma$ has exactly one row equal to $(2,4)$ then since $c_{\Gamma}(4)\leq 2k-1$ we have $\textrm{Row}_{2k+2}=(3,5)$. So $p_{\Omega} \in R_2$ and is a factor of $f$, where $\Omega=((1,2),(1,2),(2,4),(3,5),(3,5),(3,5),(4,6),(4,6))$.\\
     If none of the rows of $\Gamma$ is $(2,4)$ then since $\textrm{Row}_{2k}=(3,5)$ and $(\textrm{Row}_{2k-1},\textrm{Row}_{2k})$ is admissible we have $\textrm{Row}_{2k-1}=(3,5)$ and hence, $N_{5,2}\geq 2k-m_1+2$. Then $N_{2,1} \leq 2k-m_1-2$ and so $N_{2,2} \geq 4$. Hence $\textrm{Row}_{3}=\textrm{Row}_4=(1,2)$. We claim that $\textrm{Row}_{4k-3}=\textrm{Row}_{4k-2}=(4,6)$. If not then $c_{\Gamma}(4)\leq 3$ whereas $c_{\Gamma}(3)\geq 2k-3+2=2k-1$, a contradiction to $k \geq 3$. Since $(\textrm{Row}_{4k-m_1-1},\textrm{Row}_{4k-m_1})$ is admissible we have $(\textrm{Row}_{4k-m_1-1},\textrm{Row}_{4k-m_1}) \in \{((3,5),(3,5)),((3,5),(4,5)),((4,5),(4,5))\}$.\\
     If $(\textrm{Row}_{4k-m_1-1},\textrm{Row}_{4k-m_1})=((4,5),(4,5))$ then $p_{\Omega} \in R_1$ and is a factor of $f$, where $\Omega=((2,3),(2,3),(4,5),(4,5))$.\\
     If $(\textrm{Row}_{4k-m_1-1},\textrm{Row}_{4k-m_1})=((3,5),(3,5))$ then $\textrm{Row}_{2k+2}=(3,5)$ and in this case $p_{\Omega} \in R_3$ and is a factor of $f$, where\\
\centerline{ $\Omega=((1,2),(1,2),(1,2),(1,2),(3,5),(3,5),(3,5),(3,5),(4,6),(4,6),(4,6),(4,6))$.} \\
If $(\textrm{Row}_{4k-m_1-1},\textrm{Row}_{4k-m_1})=((3,5),(4,5))$ then for $m_1=2k-2$ we have $c_{\Gamma}(4)=2k-1$ whereas $c_{\Gamma}(3) \leq (2k-4)+1=2k-3$, a contradiction and for $m_1 \leq 2k-4$ we have $\textrm{Row}_{2k+2}=(3,5)$, then  $p_{\Omega} \in R_3$ and is a factor of $f$, where\\
\centerline{ $\Omega=((1,2),(1,2),(1,2),(1,2),(3,5),(3,5),(3,5),(3,5),(4,6),(4,6),(4,6),(4,6))$.}

 (f) Let $\textrm{Row}_1=(1,3)$. Then by a similar argument as given in (a), we see that $f$ has a factor $f_1$ such that $f_1$ is in $R_1$ or $R_2$ or $R_3$. 
      
      So by induction we conclude that $f$ is generated by the elements of degree at most $3$.
      
      \end{proof}
      
      \begin{theorem}\label{Theo4}
	The quotient $T\backslash\backslash(G/P_{\alpha_3})^{ss}_T(\mathcal
{L}_{4\varpi_3})$ is projectively normal with respect to the descent of the line bundle $\mathcal{L}_{4\varpi_3}$.
              \end{theorem}
              
              \begin{proof} 
              
              We have \[ T\backslash\backslash(G/P_{\alpha_3})_T^{ss}(\mathcal L_{4\varpi_3}) = Proj({\oplus_{k \in \mathbb{Z}_{\geq 0}}}H^0(G/P_{\alpha_3},\mathcal{L}^{\otimes k}_{4\varpi_3})^T)=Proj(\oplus_{k\in \mathbb{Z}_{\geq0}}R_k),\]where $R_k :=H^0(G/P_{\alpha_3},\mathcal{L}^{\otimes k}_{4\varpi_3})^T$. Since the quotient $T\backslash\backslash(G/P_{\alpha_3})_T^{ss}(\mathcal L_{4\varpi_3})$ is normal, in order to show that it is projectively normal we show that $R_k$ is generated by $R_1$. Let $f \in R_k$ be a standard monomial. The Young diagram associated to $f$ has the shape $p = (p_1, p_2, p_3) = (4k, 4k, 4k)$. So the Young tableau $\Gamma$ associated to this Young diagram has $4k$ rows and $3$ columns with strictly increasing rows and non-decreaing columns. Since $f$ is $T$-invariant, by \cref{zeroweight} we have, \begin{equation}\label{weight} c_\Gamma(t)=c_\Gamma(7-t) \text{ for all } 1 \leq t \leq 6. 
              \end{equation}
            Since $\bar{p_1}=0$ admissibility condition is not valid here. We also have 
          $\mbox{if } t \in \textrm{Row}_{i} \mbox{ then } 7-t \notin \textrm{Row}_{i}, \mbox{ for all } 1 \leq t \leq 6 \mbox{ and for any } i, 1 \leq i \leq 4k$, where $\textrm{Row}_{i}$ denotes the $i$th row of the tableau. For $1 \leq t \leq 6$ all the rows of $\Gamma$ contain either $t$ or $7-t$. So $c_{\Gamma}(t)=2k$ for all $1 \leq t \leq 6$. 
           
         Let $\textrm{Col}_{j}$ denotes the $j$th column of the tableau. Let $E_{i,j}$ be the $(i,j)$-th entry of the tableau and $N_{t,j} = \#\{ i | E_{i,j} = t\}$. 
         
         Note that $E_{i,1} = 1$ for all $1 \leq i \leq 2k$ and $E_{i,3} = 6$ for all $ 2k+1 \leq i \leq 4k$.

 If $E_{1,2}=4$ or $5$ then $1$, $2$ and $3$ appear $2k$ times each in the first column, a contradiction. 
 
  So, $\textrm{Row}_1 \in \{(1,2,3), (1,2,4), (1,3,5)\}$.

\textbf{Case - 1} {$\textrm{Row}_1 = (1,3,5)$} 
 
In this case we have $E_{i,1} = 2$ for all $2k+1 \leq i \leq 4k$, $E_{i,2} = 3$ for all $1 \leq i \leq 2k$, $E_{i,3} = 5$ for all $1 \leq i \leq 2k$ and $E_{i,2} = 4$ for all $2k+1 \leq i \leq 4k$. So we conclude that, $\textrm{Row}_{i} = (2,4,6)$ for all $2k+1 \leq i\leq 4k$ and $\textrm{Row}_{i} = (1,3,5)$ for all $1 \leq i\leq 2k$. Then $p_{\Omega}\in R_1$ and divides $f$, where $\Omega=((1,3,5),(1,3,5),(2,4,6),(2,4,6))$. So by induction we conclude that $f$ belongs to the subalgebra generated by $R_1$. 
 
\textbf{Case - 2} {$\textrm{Row}_1 = (1,2,4)$}

In this case we have $E_{4k,2}=5$ and $E_{4k,1}$ is either $3$ or $4$.

(a) If $E_{4k,1} =3$ then $\textrm{Row}_{4k}=(3,5,6)$. In this case $E_{2k,3}$ is either $4$ or $5$.

 If $E_{2k,3}=4$ then $E_{2k,2}=2$ and hence, $E_{i,2}=2$ for all $1 \leq i \leq 2k$ and $E_{i,3}=4$ for all $1 \leq i \leq 2k$. So we conclude that $E_{i,1}=3$ for all $2k+1 \leq i \leq 4k$ and $E_{i,3}=5$ for all $2k+1 \leq i \leq 4k$. Then $p_{\Omega} \in R_1$ and is a factor of $f$, where $\Omega=((1,2,4),(1,2,4),(3,5,6),(3,5,6))$.

 If $E_{2k,3}=5$ then $E_{2k,2}$ is either $3$ or $4$. If $E_{2k,2}=3$ then $E_{2k+1,1}=2$ and so $E_{2k+1,2}$ is either $3$ or $4$. If $E_{2k+1,2}=4$ then $p_{\Omega} \in R_1$ and is a factor of $f$, where $\Omega=((1,2,4),(1,3,5),(2,4,6),(3,5,6))$. If $E_{2k+1,2}=3$ then $c_{\Gamma}(2)+c_{\Gamma}(3) \geq 4k+1$, a contradiction.

(b) If $E_{4k,1}=4$ then $E_{2k,3} = 5$ and so $E_{2k,2}$ is either $3$ or $4$. If $E_{2k,2}=3$ then $E_{2k+1,1}=2$ and so $E_{2k+1,2}$ is either $3$ or $4$. If $E_{2k+1,2} = 4$ then $c_{\Gamma}(4)+c_{\Gamma}(5) \geq 4k+1$, a contradiction. If $E_{2k+1,2}=3$  then $p_{\Omega} \in R_1$ and is a factor of $f$, where $\Omega=((1,2,4),(1,3,5),(2,3,6),(4,5,6))$. 

\textbf{Case - 3} {$\textrm{Row}_1=(1,2,3)$}

In this case $E_{4k,2}=5$ and $E_{4k,1}$ is either $3$ or $4$.

(a) If $E_{4k,1}=4$ then $E_{2k,3}$ is either $3$ or $4$ or $5$. 

 If $E_{2k,3}=3$ then $E_{i,3}=3$ for all $1 \leq i \leq 2k$ and $E_{i,2}=2$ for all $1 \leq i \leq 2k$. Hence $E_{i,1}=4$ for all $2k+1 \leq i \leq 4k$ and $E_{i,2}=5$ for all $2k+1 \leq i \leq 4k$. Then $p_{\Omega}\in R_1$ and is a factor of $f$, where $\Omega=((1,2,3),(1,2,3),(4,5,6),(4,5,6))$. 

 If $E_{2k,3}=4$ then $E_{2k,2}=2$, $E_{i,2}=5$ for all $2k+1 \leq i \leq 4k$ and $\textrm{Row}_{2k+1}=(3,5,6)$. SThen$p_{\Omega} \in R_1$ and is a factor of $f$, where $\Omega=((1,2,3),(1,2,4),(3,5,6),(4,5,6))$.

 If $E_{2k,3}=5$ then $E_{2k,2}$ is either $3$ or $4$.\\
If $E_{2k,2}=3$ then $E_{2k+1,1}=2$ and so $E_{2k+1,2}$ is either $3$ or $4$. 
If $E_{2k+1,2}=4$ then $p_{\Omega} \in R_1$ and is a factor of $f$, where $\Omega=((1,2,3),(1,3,5),(2,4,6),(4,5,6))$. 
If $E_{2k+1,2}=3$ then for all $2 \leq i \leq 2k+1$ we have $E_{i,2}$ either $2$ or $3$. We claim that $\Gamma$ will either have a row $(1,2,4)$ or a row $(2,4,6)$. If not then $2$ and $3$ appear in all the rows of $\Gamma$, which is a contradiction, since $\textrm{Row}_{2k}=(1,3,5)$. If $(1,2,4)$ is a row of $\Gamma$ then $p_{\Omega_1}$ is a factor of $f$ and if $(2,4,6)$ is a row then $p_{\Omega_2}$ is a factor of $f$, where $\Omega_1=((1,2,4),(1,3,5),(2,3,6),(4,5,6))$ and $\Omega_2=((1,2,3),(1,3,5),(2,4,6),(4,5,6))$.\\
If $E_{2k,2}=4$ then $E_{2k+1,1}=2$ and $E_{2k+1,2}=4$. So $\textrm{Row}_{2k+1}=(2,4,6)$. Since $E_{2k,3}=5$ we have $N_{3,1}\geq 1$. So if $E_{q,1}=3$ for some $2k+2 \leq q \leq 4k-2$ we have $\textrm{Row}_{q}=(3,5,6)$. Then $p_{\Omega}\in R_1$ and is a factor of $f$, where $\Omega=((1,2,3),(1,4,5),(2,4,6),(3,5,6))$.

(b) If $E_{4k,1}=3$ then $E_{2k+1,1}=2$ and $E_{2k,3}$ is either $4$ or $5$. If $E_{2k,3}=4$ then $E_{i,2}=5$ for all $2k+1 \leq i \leq 4k$. Hence, $c_{\Gamma}(4) < 2k$, a contradiction. If $E_{2k,3}=5$ then $E_{2k,2}$ is either $3$ or $4$. If $E_{2k,2}=3$ then $c_{\Gamma}(2)+c_{\Gamma}(3) \geq 4k+1$, a contradiction. If $E_{2k,2}=4$ and in this case we have $E_{2k+1,2}=4$. Then $p_{\Omega}\in R_1$ and is a factor of $f$, where $\Omega=((1,2,3),(1,4,5),(2,4,6),(3,5,6))$.

So by induction we conclude that $f$ belongs to the subalgebra generated by $R_1$ and hence the quotient is projectively normal. 
 \end{proof}


\begin{thebibliography}{22}
\bibitem{AK} J. Akiyama, M. Kano, {\em Factors and factorizations of graphs. Proof techniques in factor theory}, Lecture Notes in Mathematics, 2031. Springer, Heidelberg, 2011.
\bibitem{sarjick} S. Bakshi, S. S. Kannan, K. V. Subrahmanyam, {\em Torus quotients of Richardson varieties in the Grassmannian}, arXiv:1901.01043.
\bibitem{BL} M. Brion, V. Lakshmibai, { \em A geometric approach to standard monomial theory}, Represent. Theory 7, 2003, 651--680 (electronic).
\bibitem{CK} J. B. Carrell, A. Kurth, {\em Normality of torus orbit closures in G/P}, J. Algebra 233, no. 1, 2000, 122--134.
%\bibitem{BG} Bruns, W., Gubeladze, J. (2009). Polytopes, Rings, and K-Theory. Springer Monographs in Mathemat-ics. Springer, Dordrecht.
%\bibitem{CC} Chevalley, C. (1955). Invariants of Finite Groups Generated by Reflections, Amer. J. Math. 77, 778-782.
\bibitem{CHK} Huah Chu, Shou-Jen Hu, Ming-chang Kang, {\em A  note on Projective normality}, Proc. Amer. Math. Soc. 139, no. 6, 2011, 1989--1992.
\bibitem{D} R. Dabrowski, { \em On normality of the closure of a generic torus orbit in G/P}, Pacific J. Math. 172, no. 2, 1996, 321--330.
%\bibitem{DN} J.M. Drezet, M. S. Narasimhan, Groupe de Picard des varieties de modules de fibrers semi-stables sur les courbes algebriques, Invent. Math. 97, 53-94. (1989).
%\bibitem{GM} I.M. Gelfand, R. MacPherson, Geometry in Grassmannians and a generalisation of the dilogarithm, Adv. in Math., 44 (1982) 279-312.
%\bibitem{GP} P. Goyal, S.K. Pattanayak, Projective Normality of G.I.T. quotients modulo Finite Group, Comm. Algebra, Vol. 45, no. 7 (2017).
%\bibitem{FH} Fulton, W., Harris, J. (1991). Representation theory, Graduate Texts in Mathematics, vol. 129, Springer-Verlag, New York. A first course; Readings in Mathematics. 
\bibitem{RH} R. Hartshorne, {\em Algebraic Geometry}, Graduate Texts in Math., 52, Springer-Verlag, New York-Heidelberg, 1977.
\bibitem{HK} C. Hausmann, and A. Knutson, {\em Polygon spaces and grassmannians}, L'Enseignement Mathematique 43, no. 1-2, 1997, 173--198.
\bibitem{How} B. J. Howard, {\em Matroids and geometric invariant theory of torus actions on flag spaces}, J. Algebra 312, no. 1, 2007, 527--541.
\bibitem{HMSV} B. J. Howard, J. Millson, A. Snowden, and R. Vakil, {\em The equations for the moduli space of n points on the line}, Duke Math. J. 146, no. 2, 2009, 175--226.
%\bibitem{Hu} Y. Hu, he Geometry and Topology of Quotient Varieties of Torus Actions, Duke Mathematical Journal 68 No. 1 (1992), 151 - 184.
\bibitem{Hum1} J. E. Humphreys, {\em Introduction to lie algebras and representation theory}, vol. 9, Springer Science \& Business Media, 2012.
\bibitem{Hum2} J. E. Humphreys, {\em Linear algebraic groups}, vol. 21, Springer Science \& Business Media, 2012.
%\bibitem{Hum} Humphreys, J.E. (1990). Reflection groups and Coxeter groups, Cambridge Advanced Studies in Mathematics, no. 29, Cambridge University Press, Cambridge.
\bibitem{K1} S. S. Kannan, {\em Torus quotients of homogeneous spaces}, Proc. Indian Acad. Sci. (Math.Sci.) 108(1), 1998, 1--12.
%\bibitem{KCP}  Kannan, S.S., Chary, B.N., Pattanayak, S.K., Torus Invariants of the Homogeneous Coordinate Ring of G/B-Connection with Coxeter Elements, Comm. Algebra, Vol. 42, no. 5 (2014).
\bibitem{KPS} S. S. Kannan, S. K. Pattanayak, Pranab Sardar.  {\em Projective normality of finite groups quotients}, Proc. Amer. Math. Soc. 137, no. 3, 2009,  863--867.
%\bibitem{KP} S.S. Kannan, S.K., Pattanayak, Projective normality of Weyl group quotients, Proc. Indian Acad. Sci.(Math. Sci), 121, no.1, pp. 1-8. (2011).
%\bibitem{KS1} S.S. Kannan, P. Sardar, Torus quotients of homogeneous spaces of the general linear group and the standard representation of certain symmetric groups. Proc. Indian Acad. Sci. (Math. Sci.) 119(1), 81-100 (2009).
\bibitem{Kap1} M. M. Kapranov, {\em Chow quotients of Grassmannians-I}, I. M. Gelfand Seminar, Adv. Soviet Math. 16, Part 2, Amer. Math. Soc., Providence, 29--110. Providence: American Mathematical Society, 1993, pp. 29-110.
\bibitem{Kap2} M. M. Kapranov, {\em Veronese curves and Grothendieck-Knudsen moduli space $\overline{M}_{0,n}$}, J. Algebraic Geom. 2, no. 2, 1993, 239--262.
\bibitem{Knu} A. Knutson, {\em Weight Varieties}, PhD thesis, M.I.T., 1996.
%\bibitem{KSS} Kraft, H., Slodowy, P., Springer, T.A. (1989). Algebraic Transformation Groups and Invariant Theory, Birkhauser.
%\bibitem{KW} Kraft, J.S., Washington, L.C., An Introduction to Number Theory with Cryptography, Published by Taylor and Francis Ltd Chapman and Hall/CRC, 2013.
\bibitem{KS} Shrawan Kumar, {\em Descent of line bundles to GIT quotients of flag varieties by maximal torus}, Transform. Groups 13, no. 3-4, 2008, 757--771.
\bibitem{LB} V. Lakshmibai, Justin Brown, {\em Flag varieties. An interplay of geometry, combinatorics, and representation theory}, 53. Hindustan Book Agency, New Delhi, 2009. 
%\bibitem{L} P. Littelmann, A generalization of the {L}ittlewood-{R}ichardson rule,   Journal of Algebra, 130 (1990), no.~2, 328-368.
\bibitem{lakshmibai-seshadri}
V. Lakshmibai, C. S. Seshadri, {\em Geometry of G/P. V}, J. Algebra 100 (1986), 462--557.
\bibitem{LP} P. Littelmann, {\em A generalization of the Littlewood-Richardson rule}, J. Algebra, 1990 no. 2, 328--368.
\bibitem{MFK} D. Mumford, J. Fogarty, F. Kirwan. {\em Geometric Invariant theory}, Springer-Verlag, 1994.
\bibitem{PN} P. E. Newstead, {\em Introduction to Moduli Problems and Orbit Spaces}, TIFR Lecture Notes, 1978.
\bibitem{spr}  T. A. Springer, {\em Linear algebraic groups}, Reprint of the 1998 second edition. Modern Birkhäuser Classics. Birkhäuser Boston, Inc., Boston, MA, 2009.
%\bibitem{Noe16} Noether, E. (1916). Der Endlichkeitssatz der Invarianten endlicher Gruppen, Math. Ann. 77, 89-92.
%\bibitem{Noe26}  Noether, E. (1926). Der Endlichkeitssatz der Invarianten endlich linearer Gruppen der Charakteristik p, Nachr. Akad. Wiss. Gottingen, 28-35.
%\bibitem{PJ} J. Petersen, Die Theorie der regul$\ddot{\mbox{a}}$ren Graphen, Acta Math. 15 (1891), 193-220.
%\bibitem{RR} S. Ramanan, A. Ramanathan, Projective normality of flag varieties and Schubert varieties. Invent. Math. 79(2), 217-224 (1985).
%\bibitem{HW} H. Weyl, The Classical Groups. Their Invariants and
%Representations, Princeton Univ. Press. (1946).
  
\end{thebibliography}
 \end{document}